\documentclass[12pt]{amsart}

\newtheorem{theorem}{Theorem}[section]
\newtheorem{lemma}[theorem]{Lemma}

\newtheorem{proposition}[theorem]{Proposition}
\newtheorem{definition}[theorem]{Definition}

\usepackage[bookmarks]{hyperref}
\usepackage{cite}
\usepackage[makeroom]{cancel}
\usepackage{graphicx,amscd,amsthm,amsfonts,amsopn,amssymb,verbatim,enumerate,xcolor}

\usepackage[letterpaper,left=1in,top=1in,right=1in,bottom=1in]{geometry}
\usepackage{xcolor}

\def\N{{\mathbb N}}

\def\nint{\mathop{\diagup\kern-13.0pt\int}}

\def\supp{{\operatorname{supp}}}

\def\bas{\begin{align*}}
\def\eas{\end{align*}}
\def\bi{\begin{itemize}}
\def\ei{\end{itemize}}

\def\emph#1{{\it #1}}

\DeclareMathOperator{\sgn}{sgn}

\DeclareMathOperator{\Tr}{Tr}
\DeclareMathOperator{\curl}{curl}
\theoremstyle{definition}
\newtheorem{remark}[theorem]{Remark}
\newtheorem{thm}[theorem]{Theorem}

\numberwithin{equation}{section}

\title[Strichartz and gallery waves for vacuum irrotational compressible Euler]{Strichartz estimates and gallery waves for the irrotational compressible Euler equation in a vacuum setting}
\author{Ovidiu-Neculai Avadanei}
\address{Department of Mathematics, California Institute of Technology, 1200 E.\ California Boulevard, Mail Code 253-37, Pasadena, CA 91125, USA}
\email{avadanei@caltech.edu}

\begin{document}
\begin{abstract}
We consider the free boundary problem for the irrotational compressible Euler equation in a physical vacuum setting. By using the irrotationality condition in the Eulerian formulation of Ifrim and Tataru, we derive a formulation of the problem in terms of the velocity potential function, which turns out to be an acoustic wave equation that is widely used in solar seismology. This paper is a first step towards understanding what Strichartz estimates are achievable for the aforementioned equation. Our object of study is the corresponding linearized problem in a model case, in which our domain is represented by the upper half-space. For this, we investigate the geodesics corresponding to the resulting acoustic metric, which have multiple periodic reflections next to the boundary. Inspired by their dynamics, we define a class of whispering gallery type modes associated to our problem, and prove Strichartz estimates for them. By using a construction akin to a wave packet, we also prove that one necessarily has a loss of derivatives in the Strichartz estimates for the acoustic wave equation satisfied by the potential function. Conversely, we prove Strichartz estimates for solutions whose degree of regularity is above the aforementioned threshold, which shows that the latter is optimal for our model problem. To the best of our knowledge, these are the first results of this kind for the irrotational compressible Euler equations in a physical vacuum.
\end{abstract}
\maketitle
\tableofcontents
\section{Introduction}
In this paper we are concerned with the evolution for the free boundary problem in a compressible gas setting. In one of the most common models, the gas occupies a domain $\Omega_t$ with boundary $\Gamma_t$ at a given time $t$, and is described by the (Eulerian) variables $(\rho,v)$, where $\rho\geq 0$ is the \text{density}, and $v$ is the \text{velocity}. The evolution is described by what are known as the \text{compressible Euler} equations
\begin{equation}\label{rawEuler}
\begin{cases}\begin{aligned}
  &\rho_t+\nabla\cdot(\rho v)=0\\
  &\rho(v_t+(v\cdot\nabla)v)+\nabla p=0
  \end{aligned}
  \end{cases}
\end{equation}
In the irrotational case, we shall momentarily see that the velocity admits a potential which solves a wave equation with respect to the acoustic metric, and the same will be true for the linearized counterpart of the Euler system.  An interesting question here is to which extent Strichartz estimates hold for the linearized problem, as well as for the full nonlinear equations. Our paper is a first step towards understanding the first part of this question in a model case. Specifically, it turns out that the geometry generated by the resulting acoustic metric gives rise to multiply reflecting geodesics along the boundary, which in turn allow us to construct solutions that force derivative losses in the estimates.

We shall  assume that the pressure is given by constitutive laws of the form
\begin{equation}
\begin{aligned}
    p(\rho)=\rho^{\kappa+1},\kappa>0.
    \end{aligned}
\end{equation}
This system can be regarded as a coupled one consisting of a wave equation for the variables $(\rho,\nabla\cdot v)$, and a transport equation for the \textit{vorticity} $\omega=\curl v$. Throughout this paper, we shall assume that our gas is irrotational, so that $\omega=0$, and $v=\nabla\phi$, where $\phi$ is a \text{potential}. This will allow us to work in a slightly different interpretation, one in which $\phi$ also solves a wave equation, which will be derived in Section \ref{s:Stream function formulation}, along with a linearized counterpart. In both interpretations, a key quantity is the \textit{speed of sound} $c_s$, which is the propagation speed for the wave components, and is given by
\begin{equation}
    c_s^2=p'(\rho).
\end{equation}

In our model, we shall allow the density $\rho$ to vanish, a scenario which corresponds to \textit{vacuum states}, so that the gas will occupy the domain $\Omega_t=\{(t,x)|\rho(t,x)>0\}$, with a moving boundary $\Gamma_t$. In contrast to the fluid case, the density will vanish on the free boundary $\Gamma_t$, which means that it can be described as
\begin{align*}
  \Gamma_t=\{(t,x)|\rho(t,x)=0\}.  
\end{align*}
In general, one expects a single stable nontrivial physical regime, known as \textit{physical vacuum}, which corresponds to the situation where $d(x,\Gamma_t)\approx c_s^2(t,x)$. Heuristically, if the sound speed has a faster decay rate, one expects the particles on the boundary to move linearly and independently, which is a regime that can only last for a short time. A slow decay rate for the speed of sound would correspond to an infinite initial acceleration of the boundary, which is another regime that cannot last for a long time. On the other hand, the physical vacuum regime allows the free boundary to move with a bounded velocity and acceleration while interacting with the interior, ensuring that linear waves with speed $c_s$ can reach the free boundary $\Gamma_t$ in finite time.

Historically, two main approaches have been adopted in fluid dynamics: an Eulerian one,  in which the reference frame is fixed, and the particles are moving, and a Lagrangian one, in which the converse is true, in the sense that the particles are stationary, while the frame is moving. Both of these interpretations have been broadly used and studied in the context of the compressible Euler equations in the full space $\mathbb{R}^d$, where the local well-posedness problem is well-studied and understood. 

For example, in the full space $\mathbb{R}^d$, the compressible Euler equations have been traditionally regarded as a symmetric hyperbolic system, which means that the results of Hughes-Kato-Marsden from \cite{Hughes-Kato-Marsden} apply here (see also Majda's work in \cite{Majda}). It follows that the problem is locally well-posed in $H^s$ in the Hadamard sense, where $s>\frac{d}{2}+1$, with the continuation  criterion
\begin{align*}
    \int_0^T\|\nabla(\rho,v)(t)\|_{L^\infty}\,dt<\infty
\end{align*}
In the irrotational case, the general Strichartz estimates of Smith-Tataru from \cite{Smith-Tataru} apply directly, which allows one to improve the local well-posedness regularity threshold to $s>\frac{d+1}{2}$, when $d=3,4,5$. In the rotational case, it is not yet known what would be the right condition to impose on the vorticity that would guarantee a similar result; see the results of Disconzi-Luo-Mazzone-Speck \cite{Disconzi-Luo-Mazzone-Speck}, Wang \cite{Wang}, Andersson-Zhang \cite{Andersson-Zhang},\cite{Andersson-Zhang2}, and Zhang \cite{Zhang}.

On the other hand, until recently, the free boundary problem corresponding to the physical vacuum had not been studied in the Eulerian setting, and all results had been obtained in Lagrangian coordinates, at high regularity, and in indirectly defined spaces.

The first ones to develop a fully Eulerian low regularity approach to the well-posedness of the compressible Euler equations in a physical vacuum were Ifrim and Tataru in \cite{Ifrim-Tataru}. In their paper, they:
\begin{itemize}
    \item proved the uniqueness of solutions, with minimal assumptions on regularity: $(\rho,v)\in\text{Lip}$. They also showed that the solutions are stable, in the sense that the distance between different solutions can be propagated in time;
    \item  developed an Eulerian Sobolev function space structure tailored to this problem;
    \item proved sharp, scale invariant energy estimates in the aforementioned spaces, while also providing a minimal continuation criterion for the regularity of solutions ($v\in\text{Lip}$);
    \item constructed regular solutions in the Eulerian framework in high regularity spaces
    \item developed a nonlinear Littlewood-Paley decomposition, which allowed them to construct rough solutions as limits of smooth solutions, while also proving continuous dependence on the initial data. Their low regularity threshold morally corresponds to the $\frac{d}{2}+1$ one obtained by Hugh-Kato-Marsden, being one derivative above scaling.
\end{itemize}
Together with Marcelo Disconzi, they later proved counterparts of these results for the relativistic Euler equations in a vacuum setting in \cite{Disconzi-Ifrim-Tataru}.

In the fluid incompressible case, analogous results were obtained by Ifrim, Pineau, Tataru, and Taylor in \cite{Ifrim-Pineau-Tataru-Taylor}. One has to note that this problem concerns the fluid cases, which poses different difficulties. For example, the density does not tend to $0$ as we approach the boundary. In particular, one must also assume the Taylor sign condition (this is necessary, as proved by Ebin in \cite{Ebin}), and include the Taylor coefficient in the estimates. We must also mention that the first ones to adopt an Eulerian approach for fluid equations in a free boundary setting were Shatah and Zheng in \cite{Shatah-Zeng,Shatah-Zeng2, Shatah-Zeng3}. However, their main focus is on the free boundary Euler equations with surface tension, and even though they also prove the existence of  a solution to the pure gravity problem in
the zero surface tension limit, their argument seems to rely on the boundedness of the curvature, which would in turn require a higher degree of regularity than in \cite{Ifrim-Pineau-Tataru-Taylor}. It is easier to draw comparisons between the latter and the memoir of Wang, Zhang, Zhao, and Zeng \cite{Wang-Zhang-Zhao-Zeng}. There, they prove the uniqueness and existence of solutions at a level of regularity which is one derivative above scaling, but their approach is restricted to graph domains with unbounded curvature. By comparison, the results of Ifrim-Pineau-Tataru-Taylor from \cite{Ifrim-Pineau-Tataru-Taylor} not only apply to more general domains with potentially more complicated geometries, but they also obtained the first proof of continuity of solutions with respect to the initial data, along with an enhanced uniqueness result, a construction of rough solutions as unique limits of smooth ones, refined low regularity energy estimates with pointwise geometric control parameters that only require very limited regularity, a new proof of the existence of smooth solutions, and an essentially scale invariant continuation criterion, akin to the one obtained by Beale-Kato-Majda for the incompressible Euler equations on the whole space. Therefore, their approach is very different from the one in \cite{Wang-Zhang-Zhao-Zeng} which in turn relies on the one adopted in the works of Alazard-Burq-Zuily \cite{Alazard-Burq-Zuily, Alazard-Burq-Zuily2}. See also de Poyferr\'e \cite{dePoyferre}.

In this paper, we shall consider the irrotational compressible Euler equations in a vacuum regime, in the Eulerian setting. 
\subsection{Notations and the conserved energy}
The material derivative $D_t$ is defined as the derivative along the particle flow, and is given by the relation
\begin{align*}
    D_t=\partial_t+v\cdot\nabla.
\end{align*}
The equations \eqref{rawEuler} can be rewritten as
\begin{equation}
\begin{cases}\begin{aligned}
  &D_t\rho+\rho\nabla\cdot v=0\\
  &\rho D_tv+\nabla p=0
  \end{aligned}
  \end{cases}
\end{equation}
If we differentiate the equation for $\rho$ once again, we obtain a wave equation
\begin{align*}
    D_t^2-\rho\nabla\cdot(\rho^{-1}p'(\rho)\nabla\rho)=\rho[(\nabla\cdot v)^2-\Tr(\nabla v(\nabla v)^T)]
\end{align*}
with propagation speed $c_s$. One can obtain a similar equation for $\nabla \cdot v$.

For the vorticity $\omega=\curl v$, one can obtain the transport equation
\begin{align*}
    D_t\omega=-\omega(\nabla v)-(\nabla v)^T\omega.
\end{align*}
These show that the compressible Euler equations can indeed be interpreted as a coupled system consisting of a pair of variables $(\rho,\nabla\cdot v)$ which solve a wave equation, and a transport equation for the vorticity $\omega=\curl v$. However, we note that the irrotationality condition will allow us to derive a different formulation for our problem, in terms of a potential function. We are going to discuss this in detail in Section 2.

The equations admit a conserved energy, given by
\begin{align*}
    E&=\int_{\Omega_t}\frac{1}{2}\rho |v|^2+\rho h(\rho)\,dx,
\end{align*}
where $h$ is the specific enthalpy, and is defined by
\begin{align*}
    h(\rho)=\int_0^\rho\frac{p(\lambda)}{\lambda^2}\,d\lambda.
\end{align*}
In a suitable setting, this energy can be interpreted as a Hamiltonian, see Chorin-Marsden \cite{Chorin-Marsden}, Ebin-Marsden \cite{Ebin-Marsden}, and Marsden-Ratiu-Weinstein \cite{Marsden-Ratiu-Weinstein}.
\subsection{The good variables}
Ifrim and Tataru recast the equation using a pair of variables that turned out to be more convenient to use. We shall briefly recall their justification.

It is not difficult to see that when $\kappa=1$, the variables $(\rho,v)$ are quite convenient to use. However, we can make a better choice when $\kappa\neq 1$. In order to see this, by using the constitutive law, we get that the sound speed is given by $c_s^2=(\kappa+1)\rho^\kappa$. We would like this to have linear behavior towards the boundary and exhibit decay, which shows that using $r=r(\rho)$ given by $r'=\rho^{-1}p'(\rho)$ might be a better choice. 

Indeed, this is tantamount to setting $\displaystyle r=\frac{\kappa+1}{\kappa}\rho^\kappa$, and it turns out that the equations can indeed be simplified, taking the form
\begin{equation}\label{Euler}
\begin{cases}\begin{aligned}
  &D_tr+\kappa r(\nabla\cdot v)=0\\
  &D_tv+\nabla r=0
  \end{aligned}
  \end{cases}
\end{equation}
These are the equations that we shall consider throughout this paper, coupled with the condition $\curl v=0$, thus taking the form
potential
\begin{equation}\label{irrotational Euler}
\begin{cases}\begin{aligned}
  &D_tr+\kappa r(\nabla\cdot v)=0\\
  &D_tv+\nabla r=0\\
  &\curl v=0.
  \end{aligned}
  \end{cases}
\end{equation}
We shall also need their linearized counterparts
\begin{equation}\label{linearized Euler}
\begin{cases}\begin{aligned}
  &D_ts+w\cdot\nabla r+\kappa s(\nabla\cdot v)+\kappa r(\nabla\cdot w)=0\\
  &D_tw+(w\cdot\nabla)v+\nabla s=0,
  \end{aligned}
  \end{cases}
\end{equation}
respectively 
\begin{equation}\label{linearized irrotational Euler}
\begin{cases}\begin{aligned}
  &D_ts+w\cdot\nabla r+\kappa s(\nabla\cdot v)+\kappa r(\nabla\cdot w)=0\\
  &D_tw+(w\cdot\nabla)v+\nabla s=0\\
  &\curl w=0.
  \end{aligned}
  \end{cases}
\end{equation}

In the next section, we shall see that they admit potential function formulations, in the sense that here exist potential functions $\phi$ and $\psi$, such that 
  \begin{equation}\label{Stream function formulation}
  \begin{aligned}
  v&=\nabla\phi\\
      D_t\phi&=\frac{|v|^2}{2}-r,
      \end{aligned}
      \end{equation}
      with the velocity potential $\phi$ also satisfying
      \begin{equation}\label{Stream function formulation wave equation}
  \begin{aligned}
      D_t^2\phi&+\nabla r\cdot\nabla\phi-\kappa r\Delta\phi=0.
      \end{aligned}
  \end{equation}
 
  The linearized velocity also admits a potential $\psi$, which turns out to satisfy the following equations:
  \begin{equation}\label{Linearized stream function formulation}
  \begin{aligned}
  w&=\nabla\psi\\
      D_t\psi&=-s.
      \end{aligned}
      \end{equation}
     Moreover, the linearized velocity potential $\psi$ also satisfies the following acoustic wave equation:
      \begin{equation}\label{Linearized stream function formulation wave equation}
  \begin{aligned}
      D_t^2\psi&-\nabla r\cdot\nabla\psi-\kappa r\Delta\psi=\kappa s(\nabla\cdot v).
      \end{aligned}
  \end{equation}
The previous acoustic wave equation also appears in helioseismology, describes solar oscillations, and is a simplified scalar equation introduced by Gizon-Barucq-Duruflé-Hanson-Leguèbe-Birch-Chabassier-Fournier-Hohage-Papini in \cite{GBDHLBCFHP}; see also \cite{AHN1,AHN2,MHFG }, as well as Muller's PhD thesis for results on the uniqueness of the associated inverse problem under various hypotheses. Even though this is a simplified version of the equation of stellar oscillations (see \cite{LO} and \cite{C}) and Galbrun's equation, the latter being introduced in \cite{Galbrun}, it still captures a large proportion of the solar dynamics; see also \cite{BMMPP,HB,BJM}, as well as \cite{HH,HLS,H}. Historically, there have been two approaches in the field of seismology: a global one, in which one solves inverse spectral problems, in order to be able "hear" the shape of the Sun (one of the results of this approach is the radially symmetric model, see \cite{CGT,Cetal}), and a local one, which relies on data such as travel times and correlations of oscillations, as well as on solving inverse boundary problems in order to construct a 2D or 3D image of the solar interior (see \cite{DJHP,Getal}). Applications of the field of helioseismology include space weather prediction, and understanding the solar cycle.
\subsection{Function spaces and energies}
The conserved energy takes the form
\begin{align*}
    E&=\int_{\Omega_t}r^{\frac{1-\kappa}{\kappa}}\left(r^2+\frac{\kappa+1}{2}rv^2\right)\,dx
\end{align*}
This motivated Ifrim and Tataru to introduce the base energy space $\mathcal{H}$ with norm
\begin{align*}
    \|(s,w)\|^2_{\mathcal{H}}=\int_{\Omega_t}r^{\frac{1-\kappa}{\kappa}}\left(s^2+\kappa rw^2\right)\,dx,
\end{align*}
which is associated to the linearized equations \eqref{linearized Euler}
for functions $(s,w)$ defined almost everywhere in the gas domain $\Omega_t$.

For higher regularity spaces, they were motivated by the second order wave equation, whose leading order operator $D_t^2-\kappa r\Delta$ is naturally associated to the acoustic metric
\begin{align*}
    g=r^{-1}dx^2
\end{align*}
This allowed them to  define the higher order Sobolev spaces $\mathcal{H}^{2k}$ for elements $(s,w)\in\mathcal{D}'(\Omega_t)\times\mathcal{D}'(\Omega_t)$, with norm
\begin{align*}
    \|(s,w)\|^2_{\mathcal{H}^{2k}}&=\sum_{\substack{|\beta|\leq 2k}}^{|\beta|-\alpha\leq k}\|r^\alpha\partial^\beta(s,w)\|^2_{\mathcal{H}},
\end{align*}
where $0\leq\alpha\leq k$. For fractional $k$, the spaces $\mathcal{H}^{2k}$ can be defined by interpolation. See \cite{Ifrim-Tataru} for the precise definition and a detailed discussion.

 Our problem admits the following scaling law:
    \begin{align*}
        r_\lambda(t,x)&=\lambda^{-2}r(\lambda t,\lambda^2x)\\
        v_\lambda(t,x)&=\lambda^{-1}v(\lambda t,\lambda^2x),
    \end{align*}
    in the sense that if $(r,v)$ is a solution, then so is $(r_\lambda,v_\lambda)$.

    We note that the critical exponent $k$ for which the homogeneous counterpart of $\mathcal{H}^{2k}$ is invariant to scaling is
    \begin{align*}
        2k_0=d+1+\frac{1}{\kappa}
    \end{align*}
 In \cite{Ifrim-Tataru}, Ifrim and Tataru defined the control parameters 
    \begin{align*}
        A=\|\nabla r-N\|_{L^\infty}+\|v\|_{\dot{C}^{\frac{1}{2}}},
    \end{align*}
    which is associated with the critical Sobolev exponent $2k_0$, and
    \begin{align*}
        B=\|\nabla r\|_{\tilde{C}^{0,\frac{1}{2}}}+\|\nabla v\|_{L^\infty},
    \end{align*}
    where
    \begin{align*}
    \|f\|_{\tilde{C}^{0,\frac{1}{2}}}=\sup_{\substack{x,y\in\Omega_t}}\frac{|f(x)-f(y)|}{r(x)^{\frac{1}{2}}+r(y)^{\frac{1}{2}}+|x-y|^{\frac{1}{2}}}
    \end{align*}
    They also introduced the phase space
    \begin{align*}
        \mathbf{H}^{2k}=\{(r,v)|(r,v)\in\mathcal{H}^{2k}\}.
    \end{align*}
    Due to the fact that the $\mathcal{H}^{2k}$ norm depends directly on $\Omega_t$ (hence on $r$), this should be regarded as an infinite dimensional manifold. Their main local well-posedness result is as follows:
    
\begin{thm}\label{Initial Theorem}[Ifrim, Tataru,\cite{Ifrim-Tataru}]
    The system \eqref{Euler} is locally well-posed in the space $\mathcal{H}^{2k}$ for $k\in\mathbb{R}$ satisfying
    \begin{align*}
    2k>2k_0+1.
    \end{align*}
\end{thm}
One of the key ingredients in their proof consisted of deriving energy estimates of the form
\begin{align*}
\|(r,v)(t)\|_{\mathcal{H}^{2s}}&\lesssim e^{\int_0^tC(A)B(u)\,du}\|(r,v)(0)\|_{\mathcal{H}^{2s}}.
\end{align*}

For the previously defined control parameters, we have the following Sobolev estimates
\begin{align*}
    A&\lesssim \|(r,v)\|_{\mathbf{H}^{2k}},k>k_0\\
    B&\lesssim \|(r,v)\|_{\mathbf{H}^{2k}},k>k_0+\frac{1}{2}.
\end{align*}
The previous threshold in Theorem \ref{Initial Theorem} is required when one tries to control $B(t)$ pointwise in time. However, this quantity appears in the energy estimates in a time-averaged manner, so a natural question is whether one can take advantage of this in order to lower the regularity threshold. This in turn motivates us to investigate Strichartz estimates for our problem.
\subsection{Main results}
 For the rest of this paper, we shall consider the simplified model 
 \begin{align*}\Omega=\{x\in\mathbb{R}^d|x_d>0\}, 
 \end{align*} with $\displaystyle r=x_d$, and $\displaystyle v=0$. The corresponding wave equation satisfied by the linearized potential $\psi$ is
\begin{equation}\label{Acoustic wave}
   \partial_t^2\psi-\kappa x_d\Delta\psi-\partial_d\psi=0.
\end{equation}
In what follows, for a given vector $a\in\mathbb{R}^d$, we shall denote by $a'$ its tangential coordinates, corresponding to the first $d-1$ coordinates, and we also consider $\displaystyle g=\kappa^{-1}x_d^{-1}dx^2$, which is the acoustic metric.

\begin{definition}
    Let $d,q,r\geq 2$, $\gamma\in\mathbb{R}$.
    \begin{enumerate}
        \item We say that $\displaystyle (q,r)$ is \textit{wave-admissible} if
    \begin{align*}
        \frac{1}{q}+\frac{d-1}{2r}\leq \frac{d-1}{4}.
    \end{align*}
    If we have equality in the previous inequality, we say that $(q,r)$ is \textit{sharp wave-admissible}.
    \item We say that $(q,r,\gamma)$ is a \text{wave Strichartz triple} if $(q,r)$ is wave-admissible and
    \begin{align*}
        \frac{1}{q}+\frac{d}{r}=\frac{d}{2}-\gamma.
    \end{align*}
    \end{enumerate}
\end{definition}
For reference, recall the Strichartz estimates for the wave equations in $\mathbb{R}^d$:
\begin{theorem}
    Let $d\geq 2$, $(q,r,\gamma)$ be wave-admissible, and let $u$ be a solution to the initial value problem
    \begin{align*}
        (\partial_t^2-\Delta)u&=0\\
        (u,u_t)(0)&=(u_0,u_1).
    \end{align*}
    Then,
    \begin{align*}
        \|u\|_{L_t^qL_x^r(\mathbb{R}\times\mathbb{R}^d)}&\lesssim \|u_0\|_{\dot{H}_x^{\gamma}(\mathbb{R}^d)}+\|u_1\|_{\dot{H}_x^{\gamma-1}(\mathbb{R}^d)}.
    \end{align*}
\end{theorem}
We also define the Littlewood-Paley tangential projectors:
\begin{definition}
     $P_{x',\lambda}$ is the Fourier multiplier defined by the relation
\begin{align*}
    \widehat{P_{x',\lambda}}f(\xi')=\left(\psi\left(\frac{\xi'}{\lambda}\right)-\psi\left(\frac{2\xi'}{\lambda}\right)\right)\hat{f}(\xi'),
\end{align*}
where $\displaystyle \psi$ is a radial smooth function supported in the region $\displaystyle \{\xi'\in\mathbb{R}^{d-1}|\|\xi'\|\leq 2\}$, with $\psi=1$ when $\displaystyle \{\xi'\in\mathbb{R}^{d-1}|\|\xi'\|\leq 1\}$.
\end{definition}
Our first main result shows that these solutions do indeed provide counterexamples to Strichartz estimates in the context of the wave equation. More precisely, we have the following

\begin{theorem}\label{Counterexample Strichartz}
Let $(q,r,\gamma)$ be wave-admissible in dimension $d\geq 2$. Then, for every $\displaystyle s\geq 0$ with $\displaystyle s<\frac{1}{2q}+\gamma+\frac{1}{2\kappa}+1$, there exists a sequence of solutions $\displaystyle(\psi^j)_{j\geq 0}$ to the initial value problem
\begin{align*}
    (\partial_t^2-\kappa x_d\Delta-\partial_d)\psi^j&=0\\
    (\psi^j,\partial_t\psi^j)(0)=(\psi^j_0,\psi^j_1),
\end{align*}
with $\displaystyle (P_{x',2^{2j}}\psi^j_0,P_{x',2^{2j}}\psi^j_1)=(\psi^j_0,\psi^j_1)$, such that 
\begin{align*}
    \sup_{\substack{j\geq 0}}(\|(\psi^j_1,\nabla_x\psi^j_0)\|_{\mathcal{H}^{2s}})\leq 1,
\end{align*}
    and whenever $\displaystyle \alpha\in \left(0,\frac{1}{2q}+\gamma+\frac{1}{2\kappa}+1-s\right)$,
we have
\begin{align*}
    \lim_{\substack{j\rightarrow\infty}}2^{-2j\alpha}\|\nabla^2\psi^j(t,\cdot)\|_{L_t^qL_x^r([0,1]\times\Omega)}=\infty.
\end{align*}

\end{theorem}
We recall that our problem admits a scaling symmetry, which for the potential $\psi$ takes the form
\begin{align*}
    \psi_\lambda(t,x)=\lambda^{-3}\psi(\lambda t,\lambda^2x)
\end{align*}
In what follows, let us consider Strichartz estimate for our problem in its intrinsic scaling. These would need to have the form
\begin{align*}
    \||\nabla_x|^{\gamma}\psi\|_{L_t^qL_x^r([0,1]\times\Omega)}\lesssim \|(\partial_t\psi(0),\nabla_x\psi(0))\|_{\mathcal{H}},
\end{align*}
which, upon imposing the scaling invariance, motivates the following
\begin{definition}
    Let $d,q,r\geq 2$, $\kappa>0$, and $\gamma\in\mathbb{R}$. We say that $\displaystyle (q,r,\gamma)$ is an \textit{Euler Strcihartz} triple if $(q,r)$ is wave-admissible and
    \begin{align*}
        \frac{1}{2q}+\frac{d}{r}=\frac{d}{2}+\frac{1}{2\kappa}+\gamma-1.
    \end{align*}
\end{definition}
Our second main result is the following
\begin{theorem}\label{Counterexample Strichartz 2}
Let $(q,r,\gamma)$ be an Euler Strichartz triple in dimension $d\geq 2$. Then, for every $\displaystyle s\geq 0$ with $\displaystyle 2s<\frac{1}{q}-2\gamma$, there exists a sequence of solutions $\displaystyle(\psi^j)_{j\geq 0}$ to the initial value problem
\begin{align*}
    (\partial_t^2-\kappa x_d\Delta-\partial_d)\psi^j&=0\\
    (\psi^j,\partial_t\psi^j)(0)=(\psi^j_0,\psi^j_1),
\end{align*}
with $\displaystyle (P_{x',2^{2j}}\psi^j_0,P_{x',2^{2j}}\psi^j_1)=(\psi^j_0,\psi^j_1)$, such that 
\begin{align*}
    \sup_{\substack{j\geq 0}}(\|(\psi^j_1,\nabla_x\psi^j_0)\|_{\mathcal{H}^{2s}})\leq 1,
\end{align*}
    and whenever $\displaystyle \alpha\in \left(0,2-\gamma-s\right)$,
we have
\begin{align*}
    \lim_{\substack{j\rightarrow\infty}}2^{-2j\alpha}\|\nabla^2\psi^j(t,\cdot)\|_{L_t^qL_x^r([0,1]\times\Omega)}=\infty.
\end{align*}
\end{theorem}
\begin{remark}
    In particular, Theorems \ref{Counterexample Strichartz} and \ref{Counterexample Strichartz 2} imply that when $\displaystyle (q,r)=(2,\infty)$, we deduce that for every $\displaystyle s\geq 0$ with $\displaystyle 2s<2k_0+\frac{1}{2}$, there exists a sequence of solutions $\displaystyle(\psi^j)_{j\geq 0}$ to the initial value problem
\begin{align*}
    (\partial_t^2-\kappa x_d\Delta-\partial_d)\psi^j&=0\\
    (\psi^j,\partial_t\psi^j)(0)=(\psi^j_0,\psi^j_1),
\end{align*}
with $\displaystyle (P_{x',2^{2j}}\psi^j_0,P_{x',2^{2j}}\psi^j_1)=(\psi^j_0,\psi^j_1)$, such that 
\begin{align*}
    \sup_{\substack{j\geq 0}}(\|(\psi^j_1,\nabla_x\psi^j_0)\|_{\mathcal{H}^{2s}})\leq 1,
\end{align*}
    and whenever $\displaystyle \alpha\in \left(0,k_0+\frac{1}{4}-s\right)$,
we have
\begin{align*}
    \lim_{\substack{j\rightarrow\infty}}2^{-2j\alpha}\|\nabla^2\psi^j(t,\cdot)\|_{L_t^2L_x^{\infty}([0,1]\times\Omega)}=\infty.
\end{align*}
In particular, this shows that the obstruction caused by gallery waves lies at least at $\frac{1}{2}$ degrees of regularity above scaling, and as we shall see in Theorem \ref{Strichartz Estimates}, this is actually the true threshold.
\end{remark}
The key idea behind the proofs of Theorems \ref{Counterexample Strichartz} and \ref{Counterexample Strichartz 2} is that for every given frequency, there exist geodesics with multiple periodic reflections next to the boundary, and exact periodic solutions to the acoustic wave equations that are highly localized in tangential frequency and propagate along them. It also turns out that there exist similar solutions to the latter, which are still highly localized in tangential frequency and propagate along the aforementioned geodesics, at least up until a certain time of coherence. In analogy with Ivanovici \cite{Ivanovici}, we shall call them \textit{gallery waves} or \textit{gallery modes}. Of course, it is natural to also investigate the behavior of gallery waves for individual frequencies and modes. Our result is as follows:

\begin{theorem}\label{Estimates for gallery modes}
   Let $d\geq 2$, $(q,r)$ a wave-admissible pair so that $(d,q,r)\neq(3,2,\infty)$, and $\displaystyle u_0$ a gallery wave corresponding to the mode $\mu_n=1+2\kappa n$, where a precise definition is going to given in Section \ref{s:Whispering gallery modes}. Then, the solution $u$ to the initial value problem
   \begin{align*}
       (\partial_t^2-\kappa x_d\Delta-\partial_d)u&=0\\
       (u,u_t)(0)&=(u_0,0),
   \end{align*}
   where $\displaystyle u_0=P_{x',2^{2j}}u_0$, we have
   \begin{align*}
       \|u\|_{L_t^qL_x^r([0,T_0]\times\Omega)}&\lesssim_{n,\kappa,q,r} \left(2^{2j}\right)^{\frac{d}{2}-1-\frac{d}{r}+\frac{1}{2\kappa}-\frac{1}{2q}}\|(0,\nabla_xu_0)\|_{\mathcal{H}(\Omega)}.
   \end{align*}
\end{theorem}

We shall actually see that the solutions to our acoustic wave equation also satisfy Strichartz estimates, up to the threshold described in our counterexamples in Theorems \ref{Counterexample Strichartz} and \ref{Counterexample Strichartz 2}. In particular, this shows that our thresholds are optimal. We have the following
\begin{theorem}\label{Strichartz Estimates}
   Let $d\geq 2$, $\kappa>0$, $(q,r)$, a wave-admissible pair so that $(d,q,r)\neq(3,2,\infty)$. Let $\displaystyle s>\frac{1}{2q}+\gamma+\frac{1}{2\kappa}+1$. If $(d,q,r)\neq(2,4,\infty)$, then the solution $u$ to the initial value problem
   \begin{align*}
       (\partial_t^2-\kappa x_d\Delta-\partial_d)u&=0\\
       (u,u_t)(0)&=(u_0,u_1),
   \end{align*}
   satisfies
   \begin{align*}
       \|\nabla_x^2u\|_{L_t^qL_x^r([0,T_0]\times\Omega)}&\lesssim_{\kappa,q,r} \|(u_1,\nabla_xu_0)\|_{\mathcal{H}^{2s}(\Omega)}.
   \end{align*}
   When $(d,q,r)=(2,4,\infty)$, we instead have the estimate
   \begin{align*}
       \|\nabla_x^2u\|_{L_t^qL_x^r([0,T_0]\times\Omega)}&\lesssim_{T_0,\kappa,q,r} \|(u_1,\nabla_xu_0)\|_{\mathcal{H}^{2s}(\Omega)}.
   \end{align*}
\end{theorem}
\begin{remark}
    In particular, Theorem \ref{Strichartz Estimates} implies that when $d\geq 4$ and $\displaystyle (q,r)=(2,\infty)$, for every $\displaystyle s\in\mathbb{R}$ with $\displaystyle 2s>2k_0+\frac{1}{2}$, the solution to the initial value problem
\begin{align*}
       (\partial_t^2-\kappa x_d\Delta-\partial_d)u&=0\\
       (u,u_t)(0)&=(u_0,u_1),
   \end{align*}
   satisfies
   \begin{align*}
       \|\nabla_x^2u\|_{L_t^2L_x^\infty([0,T_0]\times\Omega)}&\lesssim_{\kappa,q,r} \|(u_1,\nabla_xu_0)\|_{\mathcal{H}^{2s}(\Omega)},
   \end{align*}
   whenever $(2,\infty)$ is admissible (we note that this requires $d\geq 4$, which excludes the previous forbidden endpoint).
\end{remark}

One might expect to have this type of gallery waves solutions for the fully nonlinear equations as well. Another interesting question is to study gallery modes for the fully nonlinear equation, and to also analyze their nonlinear self-interactions.  Further questions to be considered are whether stationary solutions are stable upon perturbations with gallery modes, how two gallery modes could interact, and up to which time they are coherent.
\subsection{A heuristic description of gallery waves}
The behavior that we expect our gallery waves to have is consistent with the behavior of multiply reflecting geodesics on the boundary. It turns out that this phenomenon also characterizes the ray acoustics in the interior of the Sun, as noted in \cite{GBS}.

By applying the time Fourier transform to equation \eqref{Acoustic wave}, we get
\begin{align*}
    (-\kappa x_d\Delta_x'-\kappa x_d\partial_d^2-\partial_d-\tau^2)\tilde{\psi}=0
\end{align*}
Let $\xi'$ be the tangential frequency, and $\xi_d$ the normal frequency. Upon taking the Fourier transform in $x'$, we get

\begin{align*}
    (-\kappa x_d\partial_d^2-\partial_d+\kappa x_d|\xi'|^2-\tau^2)\tilde{\psi}=0.
\end{align*}

Heuristically, when $x_d\ll \tau^{-2}$, the uncertainty principle suggests that $\delta\xi_d\gg \tau^2$. Therefore, once we also restrict to $|\xi_d|\gtrsim |\xi'|$, we would have

\begin{align*}
    (-\kappa x_d\partial_d^2-\partial_d)\tilde{\psi}\approx 0
\end{align*}
This equation has two homogeneous solutions, $\tilde{\psi}=x_d^{\frac{1}{\kappa}-1}$ and $\tilde{\psi}=1$; assuming that $\kappa<1$, the finiteness of the energy
\begin{align*}
    \int x_d^{\frac{1}{\kappa}-1}(s^2+\kappa x_d|w|^2)\,dx
\end{align*}
implies that only $\tilde{\psi}=1$ is viable.  In particular, this would correspond to a solution which is roughly constant in the normal variable in the region $x_d\lesssim \tau^{-2}$. This resulting property of the solution is also the reason why we do not impose boundary conditions for the linearized acoustic equation.

Another interesting regime occurs when $x_d|\xi'|^2\gg\tau^2$. In this case, the elliptic effect will dominate in our equation for $\tilde{\psi}$, leading in particular to a very quick decay as $x_d$ grows.

We expect these two regimes to overlap when $\tau^2=x_d|\xi_d|^2$ and $\tau^2=x_d|\xi'|^2$. Fixing $x_d\approx 2^{-2j}$, the uncertainty principle implies $|\tau|\approx 2^j$, $|\xi'|\approx 2^{2j}$, and $|\xi_d|\approx 2^{2j}$. These correspond to a solution akin to a gallery mode, behaving like a bump function in both the normal and the tangential directions, and roughly periodic in time, with small variations and an oscillation period roughly equal to $2^{-j}$.This suggests that Ifrim and Tataru's result, Theorem \ref{Initial Theorem} from \cite{Ifrim-Tataru}, might be optimal at least in the frequency range where $\tau^2\lesssim  |\xi'|$. However, this leaves open the question of improving this result away from this range.

Classically, gallery modes arise as highly localized waves traveling along the boundary in convex domains in the study of the wave and Schr\"odinger equations. They were studied by Ivanovici in \cite{Ivanovici} in the context of finding counterexamples for Strichartz estimates in strictly convex domains with boundary; see also \cite{Ivanovici2, Ivanovici-Lebeau-Planchon, Ivanovici-Lebeau-Planchon2, Ivanovici-Lebeau-Planchon3, Ivanovici-Lebeau-Planchon4, Blair-Smith-Sogge,Ivanovici2025,Ivanovici-Sanz}.

 \subsection{Further historical comments}
The compressible Euler equations have been studied for a long period of time, and have also received a great deal of interest from the physical side. Allowing the density to vanish in some regions, a regime which corresponds to vacuum states, adds significantly many layers of difficulty to the problem, due to the fact that the behavior of the gas is heavily influenced by the speed of sound at the boundary. Moreover, physical vacuum is also the natural boundary condition for compressible gases. Here, one should regard the space as being divided into a particle region $\Omega_t$, and a vacuum region, which are separated by a free boundary $\Gamma_t=\partial\Omega_t$ that evolves in time. One can identify two main scenarios here, depending on the behavior of the density, or equivalently, of the sound speed $c_s$ at the free boundary:
\begin{itemize}
    \item A fluid case, in which the density and the sound speed are assumed to have a nonzero positive limit at the free boundary.
    \item A gas scenario, in which the density tends to zero near the free boundary, which is going to be the case that we shall focus on in this paper.
\end{itemize}
Fluid flows were investigated in Christodoulou-Miao \cite{Christodoulou-Miao} and Lindblad \cite{Lindblad}, while the incompressible case was also studied by Lindblad-Luo \cite{Lindblad-Luo}. In the incompressible case, we once again note the more recent results of Ifrim-Pineau-Tataru-Taylor \cite{Ifrim-Pineau-Tataru-Taylor} that we have already discussed in this introduction.

Based on the relation between the speed of sound and the distance to the boundary of a given particle, one can distinguish three cases:
\begin{enumerate}
    \item[a)] A fast decay scenario, in which $c_s\lesssim d_{\Gamma_t}$. Here, the vacuum boundary is going to be expected to evolve linearly, preventing internal waves from reaching the boundary arbitrarily fast. This means that this scenario will survive for at least a short amount of time, and that one can also apply the known results on symmetric hyperbolic systems in order to analyze the local well-posedness of the problem, see for instance DiPerna \cite{DiPerna}, Chen \cite{Chen}, and Lions \cite{Lions}, as well as Kawashima-Makino-Ukai \cite{Kawashima-Makino-Ukai}, Liu-Yang \cite{Liu-Yang}, and Chemin \cite{Chemin}. This means that this scenario does not yield a genuine free boundary problem, and as Chemin \cite{Chemin} shows in the one dimensional case, the geometry breaks down in finite time.
    \item[b)] A slow decay scenario, in which $c_s\gg d_{\Gamma_t}$. In this case, the internal waves can reach the boundary arbitrarily fast, causing the internal flow to be strongly coupled with the one of the free boundary, thus rendering this case a genuine free boundary problem. A natural choice for the decay rates is represented by the family $c_s\approx d_{\Gamma_t}^\beta$, where $\beta\in(0,1)$. Among all of these, physical and mathematical considerations suggest that there exists a single stable decay rate, corresponding to $\beta=\frac{1}{2}$. The other cases are conjectured to be unstable and to instantly transition into the stable regime. However, no rigorous results have been proved in this sense, and it is also highly likely that the scenarios $\beta<\frac{1}{2}$ and $\beta>\frac{1}{2}$ differ significantly.
    \end{enumerate}

 In the physical vacuum case, the first setting that was historically considered was the one dimensional one in Coutand-Shkoller \cite{Coutand-Shkoller} and Jang-Masmoudi \cite{Jang-Masmoudi}. While Coutand and Shkoller prove some energy estimates and provide a procedure to construct solutions, the function spaces are not completely defined, and the initial data is not directly described. This matter is addressed by Jang and Masmoudi, who introduce the Lagrangian counterparts of Ifrim and Tataru's weighted Sobolev spaces, and prove the existence and uniqueness of solutions in sufficiently regular spaces. 

 The three dimensional case has recently received significant attention. In particular, in Coutand-Lindblad-Shkoller \cite{Coutand-Lindblad-Shkoller} are formally derived in the case $\kappa=1$. In Coutand-Shkoller \cite{Coutand-Shkoller2}, existence is
 also proved via a parabolic regularization procedure, but once again, in a functional setting that is not complete; moreover, their difference bound requires more regularity on the solutions than the existence result. Independently, Jang and Masmoudi \cite{Jang-Masmoudi2} also use a parabolic regularization to prove the existence and uniqueness for solutions for arbitrary $\kappa>0$, but with a different proof of the energy estimates. However, they carry out their proofs only on the torus in the Lagrangian setting, only briefly outlining the general case.

 All of these results in the vacuum case were proved in the Lagrangian setting, and as we have already mentioned in this introduction, the first fully Eulerian results were proved by Ifrim and Tataru  \cite{Ifrim-Tataru}, who, together with Disconzi, also proved the counterparts of these results for the relativistic equations \cite{Disconzi-Ifrim-Tataru}. Later on, together with Pineau and Taylor they also obtained counterparts for these results in the incompressible fluid case \cite{Ifrim-Pineau-Tataru-Taylor}.

 Strichartz estimates for the wave and Schr\" odinger equation have a long history, and they were first derived in the case $q=r$ by Strichartz in \cite{Strichartz} for the wave and classical Schr\" odinger equations. These results were later extended to mixed $L_t^qL_x^r$ by Ginibre and Velo \cite{Ginibre-Velo} for the Schr\" odinger equation, and then independently by Ginibre-Velo \cite{Ginibre-Velo2} and Lindblad-Sogge \cite{Lindblad-Sogge}, following earlier work by Kapitanskii \cite{Kapitanski}. The remaining endpoints were settled by Keel-Tao \cite{Keel-Tao}. 
 
 In the context of wave equations with smooth variable coefficients, the first results were obtained independently by Mockenhaupt-Seeger-Sogge \cite{Mockenhaupt-Seeger-Sogge} and Kapitanskii \cite{Kapitanskii2}.For rough coefficients, the first results were obtained in dimensions $d=2$ and $d=3$ by Smith \cite{Smith} for metrics of class $C^2$ by using wave packet techniques. Smith and Sogge also showed that when the metric has H\"older regularity $C^\alpha$, with $\alpha<2$, the result might fail.
 
 The first improvements were obtained independently by Bahouri-Chemin \cite{Bahouri-Chemin3} and Tataru \cite{Tataru-I}, in which they obtained Strichartz estimates with a $\frac{1}{4}$-derivative loss.  These results were improved by Tataru in all dimensions in \cite {Tataru-II} and \cite{Tataru-III}, where he obtained Strichartz estimates with a $\frac{1}{6}$-derivative loss. In the meantime, Bahouri and Chemin also improved their own results in \cite{Bahouri-Chemin2} to a loss of derivatives slightly better than $\frac{1}{5}$. However, Smith and Tataru subsequently proved in \cite{Smith-Tataru-counter} that for general metrics of class $C^1$, a loss of $\frac{1}{6}$ is optimal. In the case of quasilinear wave equations, lossless Strichartz were obtained by by Smith and Tataru \cite{Smith-Tataru} in dimensions $d=2$ and $d=3$.

For manifolds with smooth and strictly geodesically concave boundary, Smith and Sogge \cite{Smith-Sogge} used the Melrose and Taylor parametrix to prove the Strichartz estimates in the non-endpoint cases for the corresponding wave equation. Here, the concavity condition is essential, as in its absence, there could exist multiply reflecting geodesics, as well as their limits consisting of gliding rays, which would prevent the existence of the aforementioned parametrix.

Koch, Smith, and Tataru \cite{Koch-Smith-Tataru} proved "log-loss" estimates for the spectral clusters on compact manifolds without boundary. Burq, Lebeau, and Planchon \cite{Burq-Lebeau-Planchon} obtained Strichartz type inequalities on manifolds with boundaries by using the $L^r(\Omega)$ estimates that had been proved by Smith and Sogge \cite{Smith-Sogge2} for a class of spectral operators. However, the range of indices $(q,r)$ is restricted by the admissible range for $r$ in the squarefunction estimate fort he wave equation, which controls $u$ in the space $\displaystyle L^r(\Omega,L^2((-T,T))$. For example, when $d=3$, this restricts $(q,r)$ to $\displaystyle q,r\geq 5$. Blair, Smith and Sogge \cite{Blair-Smith-Sogge} later extended the result from \cite{Burq-Lebeau-Planchon}, proving that if $\Omega$ is a compact manifold with boundary, and $(q,r,\beta)$ a triple with
\begin{align*}
    \frac{1}{q}+\frac{d}{r}=\frac{d}{2}-\beta,
\end{align*}
subject to
\begin{align*}
    \frac{3}{r}+\frac{d-1}{r}\leq\frac{d-1}{2},\text{ when }d\leq 4,\,\text{ and }\frac{1}{q}+\frac{1}{r}\leq\frac{1}{2},\text{ when }d\leq 4,
\end{align*}
then Strichartz estimates hold for solutions to the wave equations with homogeneous Dirichlet or Neumann boundary conditions, with the implicit constant depending solely on $\Omega$ and $T$.

In the context of strictly convex manifolds with boundary, counterexamples to Strichartz estimates and the extent to which they still hold were studied by Ivanovici in \cite{Ivanovici} in the case of the upper half-space. In the Friedlander model case, dispersive estimates for the general wave equations were investigated by Ivanovici, Lebeau, and Planchon in \cite{Ivanovici-Lebeau-Planchon}; for large time dispersive estimates for the Klein-Gordon and wave equations, see \cite{Ivanovici2}. For the general case, see the works of Ivanovici, Lebeau, and Planchon \cite{Ivanovici-Lebeau-Planchon2,Ivanovici-Lebeau-Planchon3}. For counterexamples to Strichartz estimates in two dimensional convex domains, see another work of the same authors \cite{Ivanovici-Lebeau-Planchon4}. For results concerning the wave equations with Neumann boundary conditions inside general strictly convex domains, see Ivanovici \cite{Ivanovici2025}. Strichartz and dispersive estimates were also obtained for the quantum ball model by Ivanovici and Sanchez-Sanz in \cite{Ivanovici-Sanz}.  

\subsection{An outline of the paper}
Here, we briefly describe the main structure of this article, along with some of the main ideas in each section. 
 \subsubsection{Velocity potential formulation for the equation} In this short section we take advantage of the irrotationality of $v$ and of its linearized counterpart $w$ in order to derive simpler, potential function formulations for both the fully nonlinear equations and the linearized one. In particular, it turns out that both potential functions solve wave equations.
 \subsubsection{The Hamiltonian of the problem and its bicharacteristics} The goal of this section is to determine the bicharacteristics of the Hamiltonian flow associated to our wave operator $\partial_t^2-\kappa x_d\Delta-\partial_d$, which will also provide us with the explicit form of the geodesics associated to the acoustic metric. This will allow us to see that there exist multiply reflecting geodesiscs next to the boundary, which shall confirm our heuristics. The explicit form of the geodesics will also enable us to motivate the choice of scales in the construction of our gallery, which are the objects of interest in this paper.
 
 \subsubsection{Whispering gallery modes} In this section we shall study the eigenvalue problem associated to our operator at a fixed tangential frequency $\xi'$. We first study the corresponding eigenfunctions heuristically by performing a WKB analysis. It turns out that the resulting ordinary differential equation has a turning point, where the behavior of the resulting solutions changes drastically. The resulting asymptotics also allow us to determine some relevant limit conditions at the boundary that will allow us to work in our energy space $\mathcal{H}$. We also describe them explicitly in terms of Laguerre polynomials and of hypergeometric functions. Our analysis here will allow us to define our gallery waves.
 \subsubsection{Proof of Theorem \ref{Estimates for gallery modes}} In this section, we reduce the proof of Theorem \ref{Estimates for gallery modes} to proving Strichartz estimates for an initial value problem in $d-1$ dimensions. We prove the aforementioned estimates by applying a stationary phase argument in order to derive dispersive estimates, from which the Strichartz estimates immediately follow using the $TT^{\ast}$ lemma.
 \subsubsection{Strichartz estimates} In this section, we prove the Strichartz estimates that are the object of Theorem \ref{Strichartz Estimates}. One key idea is that in our regularity regime, we can spectrally decompose our solution in order to reduce the proof of the estimates to the case dyadic spectral clusters. We carry out this process by a $TT^{\ast}$ argument, which also involves deriving cluster bounds for Laguerre polynomials and for the resulting $TT^\ast$ kernel. 
 \subsubsection{Construction of our counterexamples to the Strichartz estimates}
 In this section, we construct the functions that are going to serve as our counterexamples to the classical Strichartz estimates as superpositions of wave packets, which are in turn "averages" of gallery modes. Their behavior is going to roughly mimic the behavior of multiply reflecting geodesics. 
 \subsubsection{Proof of Theorems \ref{Counterexample Strichartz} and \ref{Counterexample Strichartz 2}} In this section, we use the solutions constructed in Section 7 in order to prove Theorems \ref{Counterexample Strichartz} and \ref{Counterexample Strichartz 2}. We deduce that the regularity losses in the Strichartz estimates for our equation are in fact optimal.

 \subsection{Acknowledgements} The author would like to thank Daniel Tataru and Mihaela Ifrim for very many helpful discussions.

 The author also used ChatGPT 5.6 Sol and OpenAI Codex GPT-6 Astra for exploratory discussion and editorial feedback during the preparation of this manuscript. All mathematical statements, proofs, and final formulations are humanly written and are the author’s responsibility.

 Part of this work was done while the author was participating in the "Mathematical Problems in Fluid Dynamics Reunion" at the Simons Laufer Mathematical Sciences Institute (July 2023) and the "Nonlinear Waves and Relativity" thematic program at the Erwin Schr\"odinger International Institute for Mathematics and Physics (May 2024).

 The author was also supported by the NSF
grant DMS-2054975, the Simons Foundation, as well as the NSF grant DMS-1928930 while in residence at the Simons Laufer Mathematical Sciences Institute (formerly MSRI) in Berkeley, California.
\section{Velocity potential formulation for the equations}\label{s:Stream function formulation}

In this section, we shall take advantage of the irrotationality conditions in equations \eqref{irrotational Euler} and \eqref{linearized irrotational Euler}, and derive their potential function formulations.

We recall our equations:
\begin{equation*}
\begin{cases}\begin{aligned}
  &D_tr+\kappa r(\nabla\cdot v)=0\\
  &D_tv+\nabla r=0\\
  &\curl v=0,
  \end{aligned}
  \end{cases}
\end{equation*}
as well as the linearized counterparts:
\begin{equation*}
\begin{cases}\begin{aligned}
  &D_ts+w\cdot\nabla r+\kappa s(\nabla\cdot v)+\kappa r(\nabla\cdot w)=0\\
  &D_tw+(w\cdot\nabla)v+\nabla s=0\\
  &\curl w=0.
  \end{aligned}
  \end{cases}
\end{equation*}

We claim that they admit the following reformulation:
\begin{proposition}
  There exists a potential function $\phi$ such that 
  \begin{align*}
      D_t\phi&=\frac{|v|^2}{2}-r\\
      D_t^2\phi&+\nabla r\cdot\nabla\phi-\kappa r\Delta\phi=0.
  \end{align*}
  Similarly, for the linearized equation there exists a velocity potential $\psi$ such that
  \begin{align*}
      D_t\psi&=-s\\
      D_t^2\psi&-\nabla r\cdot\nabla\psi-\kappa r\Delta\psi=\kappa s(\nabla\cdot v),
  \end{align*}
\end{proposition}
\begin{proof}
As $v$ is irrotational, there exists a potential $\phi$ such that $v=\nabla\phi$. In this case, the equation for $v$ successively becomes: 
\begin{align*}
    D_tv&=-\nabla r\iff D_t(\nabla\phi)=-\nabla r\iff \nabla(D_t\phi)-\frac{1}{2}\nabla|\nabla\phi|^2=-\nabla r\\
    \nabla(D_t\phi)&=\nabla\left(\frac{1}{2}|\nabla\phi|^2-r\right)\iff
    D_t\phi=\frac{1}{2}|\nabla\phi|^2-r+k(t)
\end{align*}
By making the substitution
\begin{align*}
    \phi\rightarrow \phi-\int_0^tk(\tau)\,d\tau,
\end{align*}
we may assume $k(t)=0$ for all $t$. This gives
\begin{align*}
    D_t\phi&=\frac{|v|^2}{2}-r
\end{align*}
Upon applying $D_t$ to both sides and using the equations for $v$ and $r$, it follows that
\begin{align*}
    D_t^2\phi=-\nabla\phi\cdot\nabla r+\kappa r\Delta\phi
\end{align*}
We linearize the potential $\phi$ in the relation $v=\nabla\phi$, which implies that $w=\nabla\psi$. By linearizing the equality $\displaystyle D_t\phi=\frac{|v|^2}{2}-r$, we successively get that
\begin{align*}
    D_t\psi&+w\cdot v=v\cdot \nabla\psi-s\iff\\
    D_t\psi&=-s
\end{align*}
By applying $D_t$ to both sides and using the equation for $s$, we infer that
\begin{align*}
    D_t^2\psi-\nabla r\cdot\nabla \psi-\kappa r\Delta \psi=\kappa s(\nabla\cdot v)
\end{align*}
This finishes the proof.
\end{proof}

As we have already briefly discussed in the introduction and we shall see in more detail in Section \ref{s:Whispering gallery modes}, if we fix a time frequency, in the region $x_d\lesssim \tau^{-2}$, where the uncertainty principle does not yet have a dominating effect, the solution to the equation will asymptotically be equal to a linear combination of two fundamental solutions: $\tilde{\psi}=x_d^{\frac{1}{\kappa}-1}$ and $\tilde{\psi}=1$. When assuming $\kappa<1$, the finiteness of the energy
\begin{align*}
    \int x_d^{\frac{1}{\kappa}-1}((\partial_t\psi)^2+\kappa x_d|\nabla\psi|^2)\,dx
\end{align*}
implies that only $\tilde{\psi}=1$ is viable.  In particular, this would correspond to a solution which is roughly constant in the normal variable in the region $x_d\lesssim \tau^{-2}$. This also shows that we do not impose boundary conditions for the linearized acoustic equation.
\section{The Hamiltonian of the problem and its bicharacteristics}\label{s:The Hamiltonian of the problem and its bicharacteristics}
In this section we shall determine the bicharacteristics of the Hamiltonian flow associated to our wave operator $\partial_t^2-\kappa x_d\Delta-\partial_d$, which will also provide us with the explicit form of the geodesics associated to the acoustic metric. This will allow us to see that there exist geodesics with multiple periodic reflections next to the boundary, which shall confirm our heuristics. This will also motivate the choice of the scales in our construction of gallery waves, which are going to be our objects of interest.

The Hamiltonian associated to the homogeneous term of second degree of our wave operator has the form
\begin{align*}
    H&=\kappa x_d\xi_d^2+\kappa x_d|\xi'|^2-\tau^2.
\end{align*}
In what follows, we shall determine its bicharacteristics, which will also give us the geodesics associated to the corresponding acoustic metric.

We consider a bicharacteristic of the form $(t,x_d,x',\tau,\xi_d,\xi')(s)$, where $s$ is the parameter, with initial data $\displaystyle (t_0,x_{d0},x'_0,\tau_0,\xi_{d0},\xi'_0)$. The equations read:
\begin{align*}
    \frac{d}{ds}(t,x_d,x',\tau,\xi_d,\xi')&=(-2\tau,2\kappa x_d\xi_d,2\kappa x_d \xi',0,-\kappa(\xi_d^2+|\xi'|^2),0).
\end{align*}
It immediately follows that
\begin{align*}
    \tau(s)&=\tau_0\\
    t(s)&=-2s\tau_0+t_0\\
    \xi'(s)&=\xi'_0.
\end{align*}
When $\xi'_0\neq 0$, we also obtain
\begin{align*}
    \frac{1}{|\xi'|}\arctan\left(\frac{\xi_d(s)}{|\xi'|}\right)- \frac{1}{|\xi'|}\arctan\left(\frac{\xi_{d0}}{|\xi'|}\right)&=-\kappa s\\
    \xi_d(s)&=|\xi'|\tan\left(\arctan\left(\frac{\xi_{d0}}{|\xi'|}\right)-\kappa|\xi'|s\right)
\end{align*}
This immediately implies that
\begin{align*}
    x_d(s)&=x_{d0}e^{2\kappa\int_0^s\xi_d(\sigma)\,d\sigma}\\
    x_d(s)&=x_{d0}e^{\int_0^s 2\kappa|\xi'|\tan\left(\arctan\left(\frac{\xi_{d0}}{|\xi'|}\right)-\kappa|\xi'|\sigma\right)\,d\sigma}\\
    &=x_{d0}e^{2\ln\left|\cos\left(\arctan\left(\frac{\xi_{d0}}{|\xi'|}\right)-\kappa|\xi'|s\right)\right|-2\ln\left|\cos\left(\arctan\left(\frac{\xi_{d0}}{|\xi'|}\right)\right)\right|}\\
    &=x_{d0}\frac{\cos^2\left(\arctan\left(\frac{\xi_{d0}}{|\xi'|}\right)-\kappa|\xi'|s\right)}{\cos^2\left(\arctan\left(\frac{\xi_{d0}}{|\xi'|}\right)\right)}\\
    &=x_{d0}\left(\frac{|\xi|^2}{2|\xi'|^2}+\cos\left(2\kappa|\xi'|s\right)\frac{|\xi'|^2-\xi_{d0}^2}{2|\xi'|^2}+\sin\left(2\kappa|\xi'|s\right)\frac{\xi_{d0}}{|\xi'|}\right).
\end{align*}
We also have
\begin{align*}
\dot{x'}&=2\kappa\xi'_0x_d\\
&=2\kappa \xi'_0x_{d0}\left(\frac{|\xi|^2}{2|\xi'|^2}+\cos\left(2\kappa|\xi'|s\right)\frac{|\xi'|^2-\xi_{d0}^2}{2|\xi'|^2}+\sin\left(2\kappa|\xi'|s\right)\frac{\xi_d}{|\xi'|}\right),
\end{align*}
hence
\begin{align*}
    x'(s)=x'_0+2\kappa \xi'_0x_{d0}\left(\frac{s|\xi|^2}{2|\xi'|^2}+\frac{\sin\left(2\kappa|\xi'|s\right)}{2\kappa|\xi'|}\frac{|\xi'|^2-\xi_{d0}^2}{2|\xi'|^2}+\frac{1-\cos\left(2\kappa|\xi'|s\right)}{2\kappa|\xi'|}\frac{\xi_{d0}}{|\xi'|}\right)
\end{align*}

On the other hand, when $\xi'_0=0$, we have
\begin{align*}
    \frac{1}{\xi_{d0}}-\frac{1}{\xi_d(s)}&=-\kappa s\\
    \xi_d(s)&=\frac{\xi_{d0}}{\kappa s\xi_{d0}+1}.
\end{align*}
We immediately obtain
\begin{align*}
    x_d(s)&=x_{d0}e^{2\kappa\int_0^s\xi_d(\sigma)\,d\sigma}\\
    x_d(s)&=x_{d0}e^{\int_0^s \frac{2\kappa \xi_{d0}}{\kappa \sigma\xi_{d0}+1}\,d\sigma}\\
    &=x_{d0}e^{2\ln\left|\kappa s\xi_{d0}+1\right|}\\
    &=x_{d0}\left(\kappa s\xi_{d0}+1\right)^2.
\end{align*}
The equation
\begin{align*}
    \dot{x'}&=0
\end{align*} also implies that 
\begin{align*}
    x'(s)=x'_0.
\end{align*}

This fully describes the bicharacteristics of the Hamiltonian flow and the geodesics associated to our acoustic metric as well. Moreover, when $\xi'\neq 0$, we can see that even though our geodesics/bicharacteristics can't be smoothly extended to the boundary, where they develop a cusp singularity, there is one meaningful way in which they can be continued (both forward and backward).

We analyze the case $\xi'\neq 0$, for a bicharacteristic starting at $\displaystyle (t_0,x_{d0},x'_0,\tau_0,\xi_{d0},\xi'_0)$. The times of collision with the boundary are of the form
 \begin{align*}
s_k&=\frac{\arctan\left(\frac{\xi_{d0}}{|\xi'|}\right)-\frac{(2k+1)\pi}{2}}{\kappa|\xi'|} 
 \end{align*}
 For $k=-1,1$ the points of collision have tangential coordinates
 \begin{align*}
     k=-1&:x'_0+2\kappa x_{d0}\xi'\left(\frac{|\xi'|^2+\xi_{d0}^2}{2|\xi'|^2}\frac{\arctan\left(\frac{\xi_{d0}}{|\xi'|}\right)-\frac{\pi}{2}}{\kappa|\xi'|}+\frac{\xi_{d0}}{2\kappa|\xi'|^2}\right)\\
     k=1&:x'_0+2\kappa x_{d0}\xi'\left(\frac{|\xi'|^2+\xi_{d0}^2}{2|\xi'|^2}\frac{\arctan\left(\frac{\xi_{d0}}{|\xi'|}\right)+\frac{\pi}{2}}{\kappa|\xi'|}+\frac{\xi_{d0}}{2\kappa|\xi'|^2}\right)
 \end{align*}
 The difference between the times of collision is given by
 $\displaystyle
    \Delta s=\frac{\pi}{\kappa|\xi'|}$, hence $\displaystyle \Delta t=\frac{2\pi\tau_0}{\kappa|\xi'|}$.
    
 The difference in the tangential coordinates is given by $\displaystyle \Delta x'=\pi x_{d0}\frac{\xi'}{|\xi'|}\frac{|\xi'|^2+\xi_{d0}^2}{|\xi'|^2}$.

 Let $t_0=0$, $\tau_0=2^j$, $|\xi'|=2^{2j}$, so that $\displaystyle \Delta t\approx 2^{-j},\Delta x'\approx 2^{-2j}$. This shows that when $t\in[0,1]$, the geodesic will hit the boundary $2^j$ times. 

We shall approximate the proportion of time that the geodesic spends in a region of the form $x_d\approx cx_{d0}$, with $\displaystyle c\ll 1$. We recall the expression for $x_d(s)$, and we get that
\begin{align*}
    x_d(s)&=x_{d0}\left(\frac{1}{2}+\cos\left(2\kappa|\xi'|s\right)\frac{1}{2}\right)
\end{align*}
We analyze the behavior close to $s_1$, where $2\kappa|\xi'|s_1=\pi$. By taking the Taylor series expansion around this point, we get that

\begin{align*}
    x_d(s)&\approx x_{d0}\left(\frac{1}{2}+\frac{1}{2}\left(-1+\frac{4\kappa^2|\xi'|^2(s-s_1)^2}{2}\right)\right)\approx x_{d0}\kappa^2|\xi'|^2(s-s_1)^2
\end{align*}
 Thus, the geodesic stays in $x_d\approx cx_{d0}$, with $c\ll 1$ for $\displaystyle\Delta s\approx \frac{c^{\frac{1}{2}}}{|\xi'|}\approx 2^{-2j}c^{\frac{1}{2}}$, hence for $\displaystyle \Delta t\approx c^{\frac{1}{2}}2^{-j}$. 
 
 Therefore, a solution of \eqref{Acoustic wave} which propagates along such a geodesic spends a positive proportion of time in here, bounded from below uniformly in $j$. In this case, due to the small time variations, the velocity component from the $B$ control parameter would have to be estimated pointwise in time by Sobolev embeddings, the $L_t^1$ component playing thus no role, with the Strichartz estimates not bring any improvement. This suggests that Ifrim and Tataru's result, Theorem \ref{Initial Theorem} from \cite{Ifrim-Tataru}, might be optimal in the frequency range where $\tau^2\lesssim  |\xi'|$.
  
\section{Whispering gallery modes}\label{s:Whispering gallery modes}
We consider the operator $\tilde{L}=-\kappa x_d\partial_d^2-\kappa\Delta_{x'}-\partial_d$. By taking the Fourier transform in the tangential variable $x'$, we obtain $\displaystyle L_{\xi'}=-\kappa x_d\partial_d^2+\kappa|\xi'|^2-\partial_d$. We have
\begin{align*}
\int L_{\xi'}f\cdot f\,dx_d&=-\int x_d^{\frac{1}{\kappa}-1}(-\kappa x_d|\xi'|^2f+\kappa x_d\partial_d^2f+\partial_df)f\,dx_d\\
&=\kappa\int x_d^{\frac{1}{\kappa}}(\partial_df)^2\,dx_d+\int x_d^{\frac{1}{\kappa}-1}f\partial_df\,dx_d\\
    &-\int x_d^{\frac{1}{\kappa}-1}f\partial_df\,dx_d+\kappa\int x_d^{\frac{1}{\kappa}}|\xi'|^2f^2\,dx_d\\
    &=\kappa\int x_d^{\frac{1}{\kappa}}(\partial_df)^2\,dx_d+\kappa\int x_d^{\frac{1}{\kappa}}|\xi'|^2f^2\,dx_d
\end{align*}
The associated quadratic form will have the form
\begin{equation}\label{Quadratic form}
\begin{aligned}
    Q(f)&=\int_0^\infty\kappa x_d^{\frac{1}{\kappa}}((\partial_df)^2+|\xi'|^2f^2)\,dx_d
\end{aligned}
\end{equation}

In what follows, we consider the eigenvalue problem for our operator $\tilde{L}$, which has the form
\begin{align*}
    (\kappa x_d\Delta_{x'}+\kappa x_d\partial_d^2+\partial_d)f&=-\lambda f
\end{align*}

Morally, studying the aforementioned eigenvalue problem corresponds to fixing the time frequency in the wave equation, which is justifiable by the fact that time frequency localization commutes with our equation. Our analysis here will allow us define gallery waves.

For the eigenvalue problem, tangential frequency localization also commutes with $\tilde{L}$, and taking the tangential Fourier transform will give

\begin{align*}
    (-\kappa x_d|\xi'|^2+\kappa x_d\partial_d^2+\partial_d)f&=-\lambda f
\end{align*}

We also note that $\lambda>0$, and we write $\lambda=\mu|\xi'|$, where $\mu>0$, for which our problem now becomes

\begin{align*}
    (\mu|\xi'|-\kappa x_d|\xi'|^2+\kappa x_d\partial_d^2+\partial_d)f&=0.
\end{align*}
\subsection{A WKB analysis for the eigenvalue problem}
We distinguish several cases:
\subsubsection{The case $\kappa x_d|\xi'|^2\gtrsim \mu|\xi'|$}
Here, we expect the elliptic effect of $\displaystyle x_d\Delta_{x'}$ to dominate. The approximate equation will have the form
\begin{align*}
    (-\kappa x_d|\xi'|^2+\kappa x_d\partial_d^2+\partial_d)f&=0.
\end{align*}
We consider an ansatz of the form
\begin{align*}
    \tilde{f}&=e^{|\xi'|(S_0(x_d)+|\xi'|^{-1}S_1(x_d))},
    \end{align*}
    where $|\xi'|$ is our fixed semiclassical parameter.  

    Inserting this ansatz into the equation, and equating the coefficients of $|\xi'|^2$ and $|\xi'|$ give the solutions
\begin{align*}
    \tilde{f}&\approx Ke^{\pm|\xi'|x_d}x_d^{-\frac{1}{2\kappa}}
\end{align*}

However, we seek solutions with the property that $\displaystyle x_d^{\frac{1}{2\kappa}}\partial_d\tilde{f}$ is square integrable, which leaves us with the option

\begin{align*}
    \tilde{f}&\approx Ke^{-|\xi'|x_d}x_d^{-\frac{1}{2\kappa}}.
\end{align*}

\subsubsection{The case $\kappa x_d|\xi'|^2\ll \mu|\xi'|$, $x_d\gtrsim 1$.} In this case, the uncertainty principle doesn't have a significant effect yet, and we expect the $\mu|\xi'|$ term to dominate, which will impose an oscillatory behavior.

The approximate equation is
\begin{align*}
    (\mu|\xi'|+\kappa x_d\partial_d^2+\partial_d)f&=0.
\end{align*}

We consider an ansatz of the form
\begin{align*}
    \tilde{f}&=e^{|\xi'|^{\frac{1}{2}}S_0(x_d)+S_1(x_d)}
\end{align*}
Inserting this ansatz into the equation, and equating the coefficients of $|\xi'|$ and $|\xi'|^{\frac{1}{2}}$ give the solutions

\begin{align*}
    \tilde{f}\approx Kx_d^{\frac{1}{4}-\frac{1}{2\kappa}}e^{\pm|\xi'|^{\frac{1}{2}}\frac{2i\sqrt{\mu}}{\sqrt{\kappa}}x_d^{\frac{1}{2}}},
\end{align*}
which is square integrable.

\subsubsection{The case $\kappa x_d|\xi'|^2\ll \mu|\xi'|$, $x_d\ll 1$, $\kappa x_d\xi_d^2\ll |\xi'|\mu$.}
We approximately have
\begin{align*}
\partial_d\tilde{f}+|\xi'|\mu\tilde{f}&=0\\
\tilde{f}&\approx Ke^{-x_d|\xi'|\mu}
\end{align*}
\subsubsection{The case $\kappa x_d|\xi'|^2\ll \mu|\xi'|$, $x_d\ll 1$, $\kappa x_d\xi_d^2\gtrsim |\xi'|\mu$.}
We approximately have
\begin{align*}
\kappa x_d\partial_d^2\tilde{f}+\partial_d\tilde{f}&=0,
\end{align*}
hence
\begin{align*}
    \tilde{f}&=K_1+K_2x_d^{1-\frac{1}{\kappa}}
\end{align*}
when $\kappa\neq 1$.

When $\kappa<1$, the second solution is not in our desired space, so we would be left with the constant.

When $\kappa=1$, we would instead have
\begin{align*}
    x_d\partial_d^2\tilde{f}+\partial_d\tilde{f}&=0\\
    \tilde{f}&=K_1+K_2\ln(x_d)
\end{align*}
Here, $\ln(x_d)$ is not in our desired space, so we will be left with the constant.

When $\kappa>1$, we shall make a choice that is going to be explicitly described in Section \ref{s:Whispering gallery modes}.

For $\xi'\neq 0$, the quadratic form corresponding to $L_{\xi'}$ is non-degenerate and positive definite. In light of our previous discussion, we shall take its domain to be

$D(Q)=\begin{cases}
    \{f|x^{\frac{1}{2\kappa}}\partial_xf,x^{\frac{1}{2\kappa}}f\in L^2(\mathbb{R}_{+})\}\cap\{f|x_d^{\frac{1}{2\kappa}-1}f\in L^2(\mathbb{R}_{+})\}\text{, when } \kappa>1\\
    \{f|x^{\frac{1}{2\kappa}}\partial_xf,x^{\frac{1}{2\kappa}}f\in L^2(\mathbb{R}_{+}),\lim_{\substack{x\rightarrow 0_+}}x^{\frac{1}{\kappa}-1}f(x)=0\}\cap\{f|x_d^{\frac{1}{2\kappa}-1}f\in L^2(\mathbb{R}_{+})\}\text{, when } \kappa<1\\
    \{f|x^{\frac{1}{2\kappa}}\partial_xf,x^{\frac{1}{2\kappa}}f\in L^2(\mathbb{R}_{+}),\lim_{\substack{x\rightarrow 0_+}}\frac{f(x)}{\ln(x)}=0\}\cap\{f|x_d^{\frac{1}{2\kappa}-1}f\in L^2(\mathbb{R}_{+})\}\text{, when } \kappa=1.
\end{cases}$

 $Q(f)$ is symmetrical, positive definite (bounded from below), so its pointwise spectrum will consist solely of positive real numbers.

By writing $\lambda(\xi')=\mu|\xi'|$, with $\mu_k>0$, our eigenvalue problem takes the form
\begin{align*}
    (\mu|\xi'|-\kappa x_d|\xi'|^2+\kappa x_d\partial_d^2+\partial_d)f&=0,
\end{align*}
where $\mu$ is independent of $|\xi'|$ (this can be immediately seen by rescaling).

\subsection{An explicit form for the solution to the eigenvalue problem}

We rewrite the equation solved by the eigenfunction as follows:
\begin{align*}
   \left(x_d\partial_d^2+\frac{1}{\kappa}\partial_d+\left(\frac{\lambda}{\kappa}-|\xi'|^2x_d\right)\right)f&=0. 
\end{align*}

This equation will have the general solution given by the formula
\begin{align*}
    f&=c_1e^{-|\xi'|x_d}U\left(\frac{1}{2}\left(\frac{1}{\kappa}-\frac{\lambda}{\kappa|\xi'|}\right),\frac{1}{\kappa},2|\xi'|x_d\right)+c_2e^{-|\xi'|x_d}L_{\frac{1}{2}\left(\frac{\lambda}{\kappa|\xi'|}-\frac{1}{\kappa}\right)}^{\frac{1}{\kappa}-1}(2|\xi'|x_d)
\end{align*}
Here, $U$ is a hypergeometric function , which, when $b$ is not an integer, is given by
\begin{align*}
   U(a,b,z)&=\frac{\Gamma(1-b)}{\Gamma(a-b+1)}\sum_{\substack{k=0}}^\infty\frac{(a)_kz^k}{(b)_kk!}+\frac{\Gamma(b-1)}{\Gamma(a)}z^{1-b}\sum_{\substack{k=0}}^\infty\frac{(a-b+1)_kz^k}{(2-b)_kk!}\,,b\notin\mathbb{Z},
   \end{align*}
   while in the case $b\in\mathbb{Z}_+$,
   \begin{align*}
   U(a,b,z)&=\frac{(-1)^b}{\Gamma(a-b+1)}\left(\frac{\log(z)}{(b-1)!}\sum_{\substack{k=0}}^\infty\frac{(a)_kz^k}{(b)_kk!}+\sum_{\substack{k=0}}^\infty\frac{(a)_k(\psi(a+k)-\psi(k+1)-\psi(k+b))z^k}{(k+b-1)!k!}\right)\\
   &-\frac{(-1)^b}{\Gamma(a-b+1)}\sum_{\substack{k=1}}^{n-1}\frac{(k-1)!z^{-k}}{(1-a)_k(b-k-1)!}\,,b\in\mathbb{Z}_{+}.
\end{align*}
$\displaystyle L^{\lambda}_{\nu}(z)$ is a generalized Laguerre polynomial given by
\begin{align*}
    L^{\lambda}_{\nu}(z)&=\frac{\Gamma(\lambda+\nu+1)}{\Gamma(\nu+1)}\sum_{\substack{k=0}}^{\infty}\frac{(-\nu)_kz^k}{\Gamma(k+\lambda+1)k!},
\end{align*}
and $\displaystyle (a)_k=\Pi_{\substack{j=1}}^k(a+j-1)$ is the Pochammer symbol.

When $\kappa\leq 1$, the discussion from the previous section forces us to choose
\begin{align*}
   f(x_d)&= e^{-|\xi'|x_d}L_{\frac{1}{2}\left(\frac{\lambda}{\kappa|\xi'|}-\frac{1}{\kappa}\right)}^{\frac{1}{\kappa}-1}(2|\xi'|x_d).
\end{align*}

Now, unless $\displaystyle \frac{1}{2}\left(\frac{\lambda}{\kappa|\xi'|}-\frac{1}{\kappa}\right)$ is an integer, as $x_d\rightarrow\infty$, we have 
\begin{align*}
    L_{\frac{1}{2}\left(\frac{\lambda}{\kappa|\xi'|}-\frac{1}{\kappa}\right)}^{\frac{1}{\kappa}-1}(2|\xi'|x_d)&\simeq e^{2|\xi'|x_d}(2|\xi'|x_d)^{-\frac{1}{2}\left(\frac{\lambda}{\kappa|\xi'|}+\frac{1}{\kappa}\right)},
\end{align*}
hence
\begin{align*}
   f(x_d)&\simeq e^{|\xi'|x_d}(2|\xi'|x_d)^{-\frac{1}{2}\left(\frac{\lambda}{\kappa|\xi'|}+\frac{1}{\kappa}\right)},
\end{align*}
which makes it impossible for $x_d^{\frac{1}{\kappa}-1}f(x_d)$ to belong to $L^2(\mathbb{R}_+)$. Thus, $\displaystyle \frac{1}{2}\left(\frac{\lambda}{\kappa|\xi'|}-\frac{1}{\kappa}\right)\in\mathbb{Z}$, hence $\mu=1+2\kappa n$ for some $n\in\mathbb{Z}_{\geq 0}$. Thus, 
\begin{align*}
   f(x_d)&= e^{-|\xi'|x_d}L_{n}^{\frac{1}{\kappa}-1}(2|\xi'|x_d), 
\end{align*}
with $\displaystyle n>-\frac{1}{2\kappa}$, and we can immediately see that this belongs to our desired class, since $L_{n}^{\frac{1}{\kappa}-1}$ is a polynomial of degree $n$.

When $\kappa>1$, both of the aforementioned fundamental solutions work and belong to the desired energy classes. However, we note that $U\left(\frac{1}{2}\left(\frac{1}{\kappa}-\frac{\lambda}{\kappa|\xi'|}\right),\frac{1}{\kappa},2|\xi'|x_d\right)$ has limited regularity in $x_d$. In fact, we can check that a function $f$ containing the aforementioned $U$-component would have a nonzero $x_d^{\frac{1}{\kappa}-1}$ Frobenius component, which would allow for very limited normal weighted regularity, severely limiting the choices of regularity exponents in our relevant $\mathcal{H}^{2s}$ classes. We shall therefore choose the Laguerre polynomial branch once again as in the $\kappa\leq 1$ case. The same reasoning as before shows that $\mu=1+2\kappa n$ for some $n\in\mathbb{Z}_{\geq 0}$, and we therefore make the choice

\begin{align*}
   f(x_d)&= e^{-|\xi'|x_d}L_{n}^{\frac{1}{\kappa}-1}(2|\xi'|x_d), 
\end{align*}
with $\displaystyle n>-\frac{1}{2\kappa}$.

In all cases, we also know that $f$ is bounded, and that its zeroes are isolated (it can be extended to a holomorphic function in the whole complex plane). Moreover, it will also follow that $f$ is continuous with respect to the parameters $|\xi'|$ and $\kappa$.
\begin{definition}
Let $n\in\mathbb{Z}_{\geq 0}$ be an integer and $\xi'\neq 0$ be a non-zero tangential frequency, the corresponding eigenvalue $\displaystyle \lambda=(1+2\kappa n)|\xi'|$, and let $\displaystyle B(n,x):=f\left(\frac{x}{|\xi'|}\right)$, where $f$ is as above. We define the set of whispering gallery modes $\displaystyle E_{n}(\Omega)$ by taking the closure in $\displaystyle \{f|x^{\frac{1}{2\kappa}-1}f\in L^2(\mathbb{R}_+\times\mathbb{R}^{d-1})\}$ of
\begin{align*}
  \left\{u(x_d,x')=\frac{1}{(2\pi)^{d-1}}\int e^{ix'\cdot\xi'}B(n,|\xi'|x_d)\hat{\varphi}(\xi')\,d\xi',\varphi\in\mathcal{S}(\mathbb{R}^{d-1})\right\} . 
\end{align*}
In particular, $\displaystyle u\in E_{\mu}(\Omega)$ is called a \textit{whispering gallery mode}.
\end{definition}
We note that every $\displaystyle u\in E_{n}(\Omega)$ solves
\begin{align*}
   (\kappa x_d\Delta+\partial_d+(1+2\kappa n)|\nabla_{x'}|)u&=0. 
\end{align*}

We shall also need the following equivalence between the norms of $\varphi\in\mathcal{S}(\mathbb{R}^{d-1})$, and its associated element $\displaystyle u_0\in E_{\mu}(\Omega)$:
\begin{lemma}\label{Equivalence lemma for gallery modes}
    Let $\displaystyle \varphi\in\mathcal{S}(\mathbb{R}^{d-1})$, and $u_0$ be defined by
    \begin{align*}
  u(x_d,x')&=\frac{1}{(2\pi)^{d-1}}\int e^{ix'\cdot\xi'}B(n,|\xi'|x)\hat{\varphi}(\xi')\psi\left(\frac{\xi'}{2^{2j}}\right)\,d\xi'
  \end{align*}
  Then, if $\psi,\psi_1,\psi_2\in C_c^{\infty}(\mathbb{R}^{d-1})$ such that $\psi\psi_1=\psi_1$, and $\psi\psi_2=\psi$, there exist constants $C_1,C_2>0$
  \begin{align*}
    C_1\left\|\psi_1\left(\frac{D_{x'}}{2^{2j}}\right)\varphi\right\|_{L_{x'}^r(\mathbb{R}^{d-1})}\leq (2^{2j})^{-\frac{1}{r}}\left\|\psi\left(\frac{D_{x'}}{2^{2j}}\right)u\left(\frac{s}{2^{2j}},\cdot\right)\right\|_{L_{x}^r(\mathbb{R}_{+}\times\mathbb{R}^{d-1})}\leq C_2\left\|\psi_2\left(\frac{D_{x'}}{2^{2j}}\right)\varphi\right\|_{L_{x'}^r(\mathbb{R}^{d-1})}
  \end{align*}
  Moreoever,
  \begin{align*}
      \left\|x_d^{\frac{1}{2\kappa}}\psi\left(\frac{D_{x'}}{2^{2j}}\right)u\right\|_{L_x^2(\mathbb{R}_{+}\times\mathbb{R}^{d-1})}&\approx (2^{2j})^{-\frac{1}{2\kappa}-\frac{1}{2}}\left\|\psi\left(\frac{D_{x'}}{2^{2j}}\right)\varphi\right\|_{L_{x'}^2(\mathbb{R}^{d-1})}\\
      \left\|x_d^{\frac{1}{2\kappa}}\psi\left(\frac{D_{x'}}{2^{2j}}\right)\nabla_xu\right\|_{L_x^2(\mathbb{R}_{+}\times\mathbb{R}^{d-1})}&\approx (2^{2j})^{-\frac{1}{2\kappa}+\frac{1}{2}}\left\|\psi\left(\frac{D_{x'}}{2^{2j}}\right)\varphi\right\|_{L_{x'}^2(\mathbb{R}^{d-1})}
  \end{align*}
  
\end{lemma}

\begin{proof}
We have 
  \begin{align*}
  \psi(2^{-2j}D_{x'})u(x_d,x')&=\frac{1}{(2\pi)^{d-1}}\int e^{ix'\cdot\xi'}B(n,|\xi'|x_d)\hat{\varphi}(\xi')\psi\left(\frac{\xi'}{2^{2j}}\right)\,d\xi'\\
  &=\frac{2^{2j(d-1)}}{(2\pi)^{d-1}}\int e^{i2^{2j}x'\cdot\eta}B(n,2^{2j}|\eta|x_d)\hat{\varphi}(2^{2j}\eta)\psi\left(\eta\right)\,d\eta
  \end{align*}
  By making the change of variables $\displaystyle x_d=s2^{-2j}$, we reduce the inequality to proving that
  \begin{align*}
    C_1\left\|\psi_1\left(\frac{D_{x'}}{2^{2j}}\right)\varphi\right\|_{L_{x'}^r(\mathbb{R}^{d-1})}\leq \left\|\psi\left(\frac{D_{x'}}{2^{2j}}\right)u(2^{-2j}s,\cdot)\right\|_{L_{x}^r(\mathbb{R}_{+}\times\mathbb{R}^{d-1})}\leq C_2\left\|\psi_2\left(\frac{D_{x'}}{2^{2j}}\right)\varphi\right\|_{L_{x'}^r(\mathbb{R}^{d-1})}
  \end{align*}
  We may assume that $\displaystyle\supp\,\psi\subset\{\eta\in\mathbb{R}^{d-1}|0<c_0\leq|\eta|\leq c_1\}$.

  We also have
  \begin{align*}
      \psi(2^{-2j}D_{x'})u(2^{2j}s,x')&=\frac{2^{2j(d-1)}}{(2\pi)^{d-1}}\int e^{i2^{2j}x'\cdot\eta}B(n,|\eta|s)\hat{\varphi}(2^{2j}\eta)\psi\left(\eta\right)\,d\eta
  \end{align*}

  We know that $B(n,0)>0$. Thus, there exists an interval $[a,b]$ such that if $s\in[a,b]$, $B(n,|\eta|s)\geq \delta>0$ for some constant $\delta$, and every $\eta\in\supp\,\psi$, hence here $\displaystyle B(n,|\eta|s)$ is positive and bounded. As in \cite{Ivanovici}, it follows that
  \begin{align*}
      \|\psi_1(2^{-2j}D_{x'})\psi(2^{-2j}D_{x'})\varphi\|_{L_{x'}^r(\mathbb{R}^{d-1})}\leq C_1\|\psi(2^{-2j}D_{x'})\varphi\|_{L_s^rL_{x'}^r([a,b]\times\mathbb{R}^{d-1})},
  \end{align*}
  where
  \begin{align*}
      C_1&=\frac{\sup_{\eta}|\psi_1(\eta)|}{\inf_{c\in[a,b]}|B(n,c)|},
  \end{align*}
  which implies that
  \begin{align*}
      \|\psi_1(2^{-2j}D_{x'})\varphi\|_{L_{x'}^r(\mathbb{R}^{d-1})}\leq C_1\|\psi(2^{-2j}D_{x'})\varphi\|_{L_s^rL_{x'}^r(\mathbb{R}_{+}\times\mathbb{R}^{d-1})}.
  \end{align*}

  The second inequality is immediate, as $\psi=\psi\psi_2$, along with the fact that $\displaystyle |B(\mu,c)|)$ is bounded immediately imply that
  \begin{align*}
      \|\psi(2^{-2j}D_{x'})\varphi\|_{L_s^rL_{x'}^r(\mathbb{R}_{+}\times\mathbb{R}^{d-1})}\leq C_2\|\psi_2(2^{-2j}D_{x'})\psi(2^{-2j}D_{x'})\varphi\|_{L_{x'}^r(\mathbb{R}^{d-1})},
  \end{align*}
  where
  \begin{align*}
      C_2=\sup_{\eta}|\psi(\eta)|\sup_{c}|B(n,c)|
  \end{align*}

  For the second estimate, Plancherel's Theorem implies that
  \begin{align*}
      \left\|x_d^{\frac{1}{2\kappa}}\psi\left(\frac{D_{x'}}{2^{2j}}\right)u\right\|^2_{L_x^2(\mathbb{R}_{+}\times\mathbb{R}^{d-1})}&\approx\int_0^\infty \int_{\mathbb{R}^{d-1}}x_d^{\frac{1}{\kappa}}\left|B(n,|\xi'|x_d)\hat{\varphi}(\xi')\psi\left(\frac{\xi'}{2^{2j}}\right)\right|^2\,d\xi'\,dx_d\\
      &= \int_{\mathbb{R}^{d-1}}\int_0^\infty x_d^{\frac{1}{\kappa}}(B(n,|\xi'|x_d))^2\,dx_d\left|\psi\left(\frac{\xi'}{2^{2j}}\right)\hat{\varphi}(\xi')\right|^2\,d\xi'\\
      &=\int_{\mathbb{R}^{d-1}}|\xi'|^{-\frac{1}{\kappa}}\int_0^\infty (x_d|\xi'|)^{\frac{1}{\kappa}}(B(n,|\xi'|x_d))^2\,dx_d\left|\psi\left(\frac{\xi'}{2^{2j}}\right)\hat{\varphi}(\xi')\right|^2\,d\xi'\\
      &=\int_{\mathbb{R}^{d-1}}|\xi'|^{-\frac{1}{\kappa}-1}\int_0^\infty y_d^{\frac{1}{\kappa}}(B(n,y_d))^2\,dy_d\left|\psi\left(\frac{\xi'}{2^{2j}}\right)\hat{\varphi}(\xi')\right|^2\,d\xi'\\
      &\approx (2^{2j})^{-\frac{1}{\kappa}-1}\left\|\psi\left(\frac{D_{x'}}{2^{2j}}\right)\varphi\right\|^2_{L_{x'}^2(\mathbb{R}^{d-1})}
  \end{align*}
  Therefore,
  \begin{align*}
      \left\|x_d^{\frac{1}{2\kappa}}\psi\left(\frac{D_{x'}}{2^{2j}}\right)u\right\|_{L_x^2(\mathbb{R}_{+}\times\mathbb{R}^{d-1})}&\approx (2^{2j})^{-\frac{1}{2\kappa}-\frac{1}{2}}\left\|\psi\left(\frac{D_{x'}}{2^{2j}}\right)\varphi\right\|_{L_{x'}^2(\mathbb{R}^{d-1})}
  \end{align*}
  The last estimate is analogous.
\end{proof}

\begin{remark}\label{Reduction}
     Let $\mu_n=1+2\kappa n$, $\displaystyle\varphi_0\in\mathcal{S}(\mathbb{R}^{d-1})$, and $u_0\in E_{n}(\Omega)$ be such that
     \begin{align*}
         u_0(x_d,x')&=\frac{1}{(2\pi)^{d-1}}\int e^{ix'\cdot \xi'}B(n,|\xi'|x_d)\hat{\varphi}(\xi')\,d\xi'.
     \end{align*}
     In order to prove Theorem \ref{Estimates for gallery modes}, we may reduce the problem to the study of Strichartz type estimates for a problem with initial data $\varphi_0$. More precisely,  if u solves 
     \begin{align*}
         (\partial_t^2-\kappa x_d\Delta-\partial_d)u&=0\\
         (u,\partial_tu)(0)&=(P_{x',2^{2j}}u_0,0)
     \end{align*} then in order to show that the gallery modes satisfy the estimates from Theorem \ref{Estimates for gallery modes}, it suffices to show that the solution to
     \begin{align*}
         (\partial_t^2+(1+2\kappa n)|\nabla_{x'}|)\varphi&=0\\
         (\varphi,\partial_t\varphi)(0)=(P_{x',2^{2j}}\varphi_0,0)
     \end{align*}
     satisfies
     \begin{align*}
   \|f\|_{L_t^qL_{x'}^r([0,T]\times\mathbb{R}^{d-1})}\lesssim \mu_n^{-\frac{1}{2q}}2^{2j\left((d-1)\left(\frac{1}{2}-\frac{1}{r}\right)-\frac{1}{2q}\right)}\|P_{x',2^{2j}}f_0\|_{L_{x'}^2(\mathbb{R}^{d-1})},
\end{align*}
\end{remark}
    \begin{proof}
        Let $\varphi$ be the solution of the initial value problem
        \begin{align*}
         (\partial_t^2+(1+2\kappa n)|\nabla_{x'}|)\varphi&=0\\
         (\varphi,\partial_t\varphi)(0)=(P_{x',2^{2j}}\varphi_0,0)
     \end{align*}
     We define
     \begin{align*}
         u(t,x)&=\frac{1}{(2\pi)^{d-1}}\int e^{ix'\cdot \xi'}B(n,|\xi'|x_d)\hat{\varphi}(t,\xi')\,d\xi'.
     \end{align*}
     Here, $\hat{\varphi}$ is the Fourier transform with respect to the tangential variables.
     
     We claim that $u$ is the solution to the problem
     \begin{align*}
         (\partial_t^2-\kappa x_d\Delta-\partial_d)u&=0\\
         (u,\partial_tu)(0)&=(P_{x',2^{2j}}u_0,0).
     \end{align*} 
     Indeed, by using that $\displaystyle (\kappa x_d\partial_d^2+\partial_d)B(n,|\xi'|x_d)=(\kappa x_d|\xi'|^2-(1+2\kappa n)|\xi'|)B(n,|\xi'|x_d)$, we find that
     \begin{align*}
         (\partial_t^2-\kappa x_d\Delta_{x'}-\kappa x_d\partial_d^2-\partial_d)u&=\frac{1}{(2\pi)^{d-1}}\int e^{ix'\cdot \xi'}B(n,|\xi'|x_d)(\partial_t^2\hat{\varphi}(t,\xi')+\mu|\xi'|\hat{\varphi})\,d\xi'=0.
     \end{align*}
     Here, we have used the fact that \begin{align*}
       (\partial_t^2+(1+2\kappa n)|\nabla_{x'}|)\varphi=0,  \end{align*} which implies that \begin{align*} \partial_t^2\hat{\varphi}(t,\xi')+(1+2\kappa n)|\xi'|\hat{\varphi}=0.
       \end{align*}
       The initial data conditions are also clearly true.
       
       The rest of the proof immediately follows from Lemma \ref{Equivalence lemma for gallery modes}.
    \end{proof}

\section{Proof of Theorem \ref{Estimates for gallery modes}}\label{s:Proof of Theorem 3}

From Corollary \ref{Reduction}, we know that it suffices to prove the following

\begin{proposition}
    If $\mu_n=1+2\kappa n$ and $f$ is a solution of
    \begin{align*}
(\partial_t^2+\mu_n|\nabla_{x'}|)f&=0\\
(f,\partial_tf)(0)&=(P_{x', 2^{2j}}f_0,0),
\end{align*}
then
\begin{align*}
    \|f\|_{L_t^qL_{x'}^r}&\lesssim \mu_n^{-\frac{1}{2q}}2^{2j\left((d-1)\left(\frac{1}{2}-\frac{1}{r}\right)-\frac{1}{2q}\right)}\|P_{x', 2^{2j}}f_0\|_{L_{x'}^2}.
\end{align*}
\end{proposition}
\begin{proof}
Let us first take $\displaystyle f_0=\varphi_0\in\mathcal{S}(\mathbb{R}^{d-1})$. The solution to our problem  has the form
\begin{align*}
    f&=\frac{1}{(2\pi)^{d-1}}\int e^{ix'\cdot \xi'}\cos(t\mu_n^{\frac{1}{2}}|\xi'|^{\frac{1}{2}})\psi\left(\frac{\xi'}{2^{2j}}\right)\hat{f}_0(\xi')\,d\xi'\\
    &=\frac{1}{2(2\pi)^{d-1}}\int e^{ix'\cdot \xi'}e^{-it\mu_n^{\frac{1}{2}}|\xi'|^{\frac{1}{2}}}\psi\left(\frac{\xi'}{2^{2j}}\right)\hat{f}_0(\xi')\,d\xi'\\
    &+\frac{1}{2(2\pi)^{d-1}}\int e^{ix'\cdot \xi'}e^{it\mu_n^{\frac{1}{2}}|\xi'|^{\frac{1}{2}}}\psi\left(\frac{\xi'}{2^{2j}}\right)\hat{f}_0(\xi')\,d\xi'\\
    &:=f_1+f_2
\end{align*}
We shall only prove the bound for $f_1$, since the other one is similar. This can be rewritten as

\begin{align*}
    f_1&=\frac{1}{(2\pi)^{d-1}}\int e^{ix'\cdot \xi'}e^{-it\mu_n^{\frac{1}{2}}|\xi'|^{\frac{1}{2}}}\psi\left(\frac{\xi'}{2^{2j}}\right)\hat{f}_0(\xi')\,d\xi'\\
    &=\frac{2^{2j(d-1)}}{(2\pi)^{d-1}}\int e^{i2^{2j}x'\cdot \eta}e^{-it2^{j}\mu_n^{\frac{1}{2}}|\eta|^{\frac{1}{2}}}\psi\left(\eta\right)\hat{f}_0(2^{2j}\eta)\,d\eta\\
    &=\frac{2^{2j(d-1)}}{(2\pi)^{d-1}}\int e^{it\mu_n^{\frac{1}{2}}2^{j}\left(\frac{2^jx'}{t\mu_n^{\frac{1}{2}}}\cdot \eta-|\eta|^{\frac{1}{2}}\right)}\psi\left(\eta\right)\hat{f}_0(2^{2j}\eta)\,d\eta
\end{align*}
Let
\begin{align*}
G(j,\eta)&=|\eta|^{\frac{1}{2}}\,,\lambda=t\mu_n^{\frac{1}{2}}2^{j}\,,z=\frac{2^jx'}{t\mu_n^{\frac{1}{2}}}\\
    J(z,j,\lambda)&:=\int e^{i\lambda(z\cdot\eta-G(j,
    \eta))}\psi(\eta)\,d\eta\\
    \gamma_{d-1}(j,\lambda)&=\sup_{\substack{z\in\mathbb{R}^{d-1}}}|J(z,j,\lambda)|
\end{align*}

When $\lambda$ is small, the bounds follow immediately. For large $\lambda$, we are going to apply the following
\begin{lemma}[Lemma 3.3,\cite{TaoNotes}]\label{Stationary phase}
    Let $a$ be a bump function, and let $\displaystyle \phi:\mathbb{R}^d\rightarrow\mathbb{R}$ be smooth and have a stationary point at $x_0$ with $\displaystyle\det(\nabla^2\phi(x_0))\neq 0$. If $\phi$ has no other stationary points on the support of $a$, then there
exist constants $c_0$, $c_1$,$\dots$ with each $c_n$ depending (in some explicit fashion) only on finitely many derivatives of $a$, $\phi$ at $x_0$, such that we have the asymptotic formula
\begin{align*}
    \int_{\mathbb{R}^d} e^{i\lambda \phi(x)}a(x)\,dx&=\sum_{\substack{n=0}}^N c_n\lambda^{-n-\frac{d}{2}}+O_{a,\phi,\lambda}(\lambda^{-N-\frac{d}{2}-1}),
\end{align*}
for all $N\geq 0$. Furthermore,
\begin{align*}
    c_0=e^{\frac{i\pi\sgn(\nabla^2\phi(x_0))}{4}}e^{i\lambda\phi(x_0)}\sqrt{\frac{2\pi}{|\det\nabla^2\phi(x_0)|}}a(x_0).
\end{align*}
\end{lemma}

The critical point of the phase is given by
\begin{align*}
    z&=\frac{\eta}{2|\eta|^{\frac{3}{2}}},
    \end{align*}
   If $\eta(z)\in\supp\psi$, the Hessian of $G$ at the critical point is
    \begin{align*}
    (\nabla^2G(j,\eta))_{ij}&=\frac{1}{2|\eta|^{\frac{3}{2}}}\left(\delta_{ij}-\frac{3\eta_i\eta_j}{2|\eta|^2}\right),
\end{align*}
which is nondegenerate. Lemma \ref{Stationary phase} implies that 
\begin{align*}
    J(z,j,\lambda)&\approx \left(\frac{2\pi}{\lambda^{\frac{1}{2}}}\right)^{d-1}\frac{e^{-\frac{i\pi}{4}\sgn \nabla^2G(\eta(z))}}{\sqrt{|\det \nabla^2G(\eta(z))|}}\psi(\eta(z)),
    \end{align*}
    hence
    \begin{align*}
    |J(z,j,\lambda)|&\approx (t2^{j}\mu_n^{\frac{1}{2}})^{-\frac{d-1}{2}}|\eta(z)|^{\frac{3(d-1)}{4}}, 
\end{align*}
and as $\eta(z)$ is in $\supp\psi$ (which is a neighbourhood of a
the unit sphere) and away from $0$, 
\begin{align*}
  |J(z,j,\lambda)|&\lesssim t^{-\frac{d-1}{2}}2^{-j\left(\frac{d-1}{2}\right)}\mu_n^{-\frac{d-1}{4}}
\end{align*}
When $\eta(z)\notin\supp\psi$, the same bound follows by repeated integrations by parts.

This immediately implies that we may take
\begin{align*}
  \gamma_{d-1}(j,\lambda):= t^{-\frac{d-1}{2}}2^{-j\left(\frac{d-1}{2}\right)}\mu_n^{-\frac{d-1}{4}}
\end{align*}
We are now going to use the following
\begin{lemma}\label{Strichartz lemma}[c.f.Lemma 2.2,\cite{Ivanovici} and \cite{Keel-Tao}]
Let $\alpha>0$ and $(q,r)$ be such that $\displaystyle \frac{1}{q}+\frac{\alpha}{r}\leq\frac{\alpha}{2}$, with $q\geq 2$, in dimension $n$, and $(\alpha,q,r)\neq(1,2,\infty)$. Let $\displaystyle \beta=n\left(\frac{1}{2}-\frac{1}{r}\right)-\frac{1}{q}$. If the solution $\displaystyle u=e^{-\frac{it}{h}G}\psi(hD)u_0$ of the initial value problem
\begin{align*}
    \frac{h}{i}\partial_tu-G\left(\frac{hD}{i}\right)u&=0\\
    u_{|t=0}&=\psi(hD)u_0
\end{align*}satisfies the dispersive estimates
\begin{align*}
    \left\|e^{-\frac{it}{h}G}\psi(hD)u_0\right\|_{L_x^\infty(\mathbb{R}^n)}&\lesssim\frac{1}{(2\pi h)^n}\gamma_{n,h}\left(\frac{t}{h}\right)\|\psi(hD)u_0\|_{L_x^1(\mathbb{R}^n)}
\end{align*} for some function $\displaystyle \gamma_{n,h}:\mathbb{R}\rightarrow[0,\infty)$ , and every $\displaystyle t\in(0,T_0]$, then there exists some $C>0$ independent of $h$, such that the following inequality holds
\begin{align*}
   h^\beta \left\|e^{-\frac{it}{h}G}\psi(hD)u_0\right\|_{L_t^qL_x^r((0,T_0]\times\mathbb{R}^n)}&\leq C\left(\sup_{\substack{s\in\left(0,\frac{T_0}{h}\right)}}s^\alpha\gamma_{n,h}(s)\right)^{\frac{1}{2}-\frac{1}{r}}\|u_0\|_{L_x^2(\mathbb{R}^n)}
\end{align*}
\end{lemma}
We note that the equality case is Lemma 2.2 in \cite{Ivanovici}, and that the estimate for the other admissible pairs follows by the same $TT^\ast$ argument presented there. On the other hand, when $\alpha>1$, the estimate for the sharp endpoint $q=2$, $r=\frac{2\alpha}{\alpha-1}$ follows from the $TT^\ast$ argument in \cite{Keel-Tao}, while for the remaining admissible pairs $(q,r)$ with $q=2$, the estimate follows from Bernstein's inequalities.

We take $n=d-1$ and $\alpha=\frac{d-1}{2}$. We thus get that
\begin{align*}
    \|f(t)\|_{L_{x'}^{\infty}(\mathbb{R}^{d-1})}&\lesssim 2^{\frac{3j(d-1)}{2}}\gamma_{d-1}(j,\lambda)\|f_0\|_{L_{x'}^1(\mathbb{R}^{d-1})}\\
    &\lesssim t^{-\frac{d-1}{2}}2^{j\left(d-1\right)}\mu_n^{-\frac{d-1}{4}}\|f_0\|_{L_{x'}^1(\mathbb{R}^{d-1})}
\end{align*}
By interpolation,
\begin{align*}
    \|f_1(t)\|_{L_{x'}^{r}(\mathbb{R}^{d-1})}&\lesssim t^{-\frac{d-1}{2}\left(1-\frac{2}{r}\right)}2^{j\left(d-1\right)\left(1-\frac{2}{r}\right)}\mu_n^{-\frac{d-1}{4}\left(1-\frac{2}{r}\right)}\|f_0\|_{L_{x'}^{r'}(\mathbb{R}^{d-1})}
\end{align*}
Setting $\tau=t\sqrt{\mu_n}$, Lemma \ref{Strichartz lemma} now implies that
\begin{align*}
   \|f_1\|_{L_\tau^qL_{x'}^r([0,T]\times\mathbb{R}^{d-1})}\lesssim 2^{2j\left((d-1)\left(\frac{1}{2}-\frac{1}{r}\right)-\frac{1}{2q}\right)}\|P_{x',2^{2j}}f_0\|_{L_{x'}^2(\mathbb{R}^{d-1})},
\end{align*}
hence
\begin{align*}
   \|f_1\|_{L_t^qL_{x'}^r([0,T]\times\mathbb{R}^{d-1})}\lesssim_n 2^{2j\left((d-1)\left(\frac{1}{2}-\frac{1}{r}\right)-\frac{1}{2q}\right)}\|P_{x',2^{2j}}f_0\|_{L_{x'}^2(\mathbb{R}^{d-1})},
\end{align*}
which finishes the proof.
\end{proof}
Now, Lemma \ref{Equivalence lemma for gallery modes}, implies that
\begin{align*}
   \|f_1\|_{L_t^qL_{x}^r([0,T]\times((\mathbb{R}_+\times\mathbb{R}^{d-1})))}\lesssim_n (2^{2j})^{\left(\frac{d}{2}-1-\frac{d}{r}+\frac{1}{2\kappa}-\frac{1}{2q}\right)\left(\frac{1}{2}-\frac{1}{r}\right)+\frac{1}{2\kappa}-1}\|(0,\nabla_xP_{x',2^{2j}}f_0)\|_{\mathcal{H}},
\end{align*}
which finishes the proof of Theorem \ref{Estimates for gallery modes}.
\section{Strichartz Estimates}
In this section we shall prove Theorem \ref{Strichartz Estimates}. One key idea is that in our regularity regime, we can spectrally decompose our solution in order to reduce the proof of the estimates to the case dyadic spectral clusters. We carry out this process by a $TT^{\ast}$ argument, which also involves deriving cluster bounds for Laguerre polynomials and for the resulting $TT^\ast$ kernel. 

Let $\alpha=\frac{1}{\kappa}-1$. We also define $\displaystyle h_n=\frac{\Gamma(n+\alpha+1)}{n!\cdot2^{\alpha+1}}$, and
$\displaystyle b_n(X)=h_n^{-\frac{1}{2}}B(n,X)$. In this way, $b_n$ will form an orthonormal basis of our relevant Hilbert space. We shall also say that $n\sim N$ if $N\leq n+1<2N$, where $N\geq 1$ is a dyadic positive integer.
We define $\displaystyle a_{\kappa,l}=\begin{cases}
    \max\left\{\frac{1}{\kappa},1-\frac{1}{\kappa}\right\},l=0,\\
    \frac{1}{\kappa}+2l,l=\overline{1,2}.
\end{cases}$

We shall first analyze dyadic clusters of gallery modes. We have the following
\begin{lemma}\label{Uniform Lemma}
    Let $\kappa>0$. For every integer $l=\overline{0,2}$, we have the bound
    \begin{align*}
        \sup_{X\geq 0}\sum_{\substack{n\sim N}}|b_n^{(l)}(X)|^2&\lesssim_\kappa N^{a_{\kappa,l}}.
    \end{align*}
    In particular, for every sequence $(c_n)_{n\simeq N}$, we have
    \begin{align*}
        \|\sum_{\substack{n\sim N}}c_nb^{(l)}_n(X)\|_{L_X^\infty}&\lesssim_\kappa N^{\frac{a_{\kappa,l}}{2}}\left(\sum_{\substack{n\sim N}}|c_n|^2\right)^{\frac{1}{2}}. 
    \end{align*}
\end{lemma}
\begin{proof}
    We first analyze the case $\kappa\leq 1$. In this case, it follows that $\alpha\geq 0$. We have the identity $\frac{d^j}{dX^j}L_n^{\alpha}(X)=(-1)^jL^{\alpha+j}_{n-j}(X)$. This in turn implies the relation
    \begin{align*}
        B_n^{(l)}(X)&=(-1)^le^{-X}\sum_{\substack{j=0}}^l2^j\binom{l}{j}L^{\alpha+j}_{n-j}(2X).
    \end{align*}
    For every $\beta\geq 0$, we have the classical estimate (see \cite{DLMF})
    \begin{align*}
        e^{-X}|L_k^\beta(2X)|\leq L_k^\beta(0)=\frac{\Gamma(k+\beta+1)}{\Gamma(\beta+1)k!}
    \end{align*}
    For a fixed $j\in\{0,1,2\}$, we also note the estimate
    \begin{align*}
      L^{\alpha+j}_{n-j}(0)&\leq \frac{\Gamma(n+\alpha+1)}{\Gamma(\alpha+j+1)(n-j)!}\lesssim_{\kappa} (1+n)^{\alpha+j}.
    \end{align*}
    This implies that
    \begin{align*}
        \|B_n^{(l)}\|_{L_X^\infty}&\lesssim_{\kappa} (1+n)^{\alpha+l}.
    \end{align*}
    
    We also have 
    \begin{align*}
      h_n&=2^{-\alpha-1}\frac{\Gamma(n+\alpha+1)}{(n)!}\simeq_{\kappa} (1+n)^{\alpha}.
    \end{align*}
    This implies that
    \begin{align*}
        \|b_n^{(l)}\|_{L_X^\infty}&\lesssim_{\kappa} (1+n)^{\frac{\alpha}{2}+l}.
    \end{align*}
    We immediately get that
    \begin{align*}
    \sum_{\substack{n\sim N}}|b_n^{(l)}(X)|^2&\lesssim_\kappa N\cdot N^{\alpha+2l}=N^{\frac{1}{\kappa}+2l}.
    \end{align*}
    The second estimate follows immediately by applying the Cauchy-Buniakowski-Schwartz inequality.

    We now analyze the case $\kappa>1$, which implies that $\alpha\in(-1,0)$. Let $\nu_n:=4n+2\alpha+2$. From Muckenhoupt \cite{Muckenhoupt} and Askey-Wainger \cite{Askey-Wainger}, for $B_n(X)=e^{-X}L_n^\alpha(2X)$, we have the bounds
    
    $\displaystyle |B_n(X)|\lesssim_\alpha\begin{cases}\nu_n^\alpha, 0\leq 2X\leq \nu_n^{-1},\\
    \nu_n^{\frac{\alpha}{2}-\frac{1}{4}}(2X)^{-\frac{\alpha}{2}-\frac{1}{4}},\nu_n^{-1}<2X\leq \frac{\nu_n}{2},\\
    [\nu_n(\nu_n^{\frac{1}{3}}+|2X-\nu_n|)]^{-\frac{1}{4}},\frac{\nu_n}{2}<2X\leq \frac{3\nu_n}{2},\\
      \nu_n^{\frac{\alpha}{2}}e^{-cX}, \frac{3\nu_n}{2}<2X.  
    \end{cases}$
    
    Since $h_n\simeq_\alpha N^\alpha$, and $n\simeq N$, we get that

    $\displaystyle |b_n(X)|\lesssim_\alpha\begin{cases}N^{\frac{\alpha}{2}}, 0\leq X\lesssim N^{-1},\\
    N^{-\frac{1}{4}}X^{-\frac{\alpha}{2}-\frac{1}{4}},N^{-1}\lesssim X\lesssim N, \\
    N^{-\frac{1}{4}-\frac{\alpha}{2}}[(N^{\frac{1}{3}}+|2X-\nu_n|)]^{-\frac{1}{4}},X\sim N,\\
     e^{-c_\alpha X}, N\lesssim X.  
    \end{cases}$
    
We first analyze the regime $X\lesssim N^{-1}$. We clearly have
\begin{align*}
     \sum_{\substack{n\sim N}}|b_n(X)|^2\lesssim_\alpha N^{1+\alpha}.  
    \end{align*}
We now analyze the regime $N^{-1}\lesssim X\lesssim N$, in which case
    \begin{align*}
     |b_n(X)|^2\lesssim_\alpha N^{-\frac{1}{2}}X^{-\alpha-\frac{1}{2}}\lesssim_\alpha N^{-\frac{1}{2}}N^{\max\left\{\alpha+\frac{1}{2},-\alpha-\frac{1}{2}\right\}}.  
    \end{align*}
    Thus,
     \begin{align*}
     \sum_{\substack{n\sim N}}|b_n(X)|^2\lesssim_\alpha N^{\max\left\{\alpha+1,-\alpha\right\}}.  
    \end{align*}

    We analyze the turning point regime $X\sim N$. Here we have
    \begin{align*}
     |b_n(X)|^2\lesssim_\alpha N^{-\frac{1}{2}-\alpha}(N^{\frac{1}{3}}+|2X-\nu_n|)^{-\frac{1}{2}},  
    \end{align*}
    hence
    \begin{align*}
     \sum_{\substack{n\sim N}}|b_n(X)|^2&\lesssim_\alpha \sum_{\substack{n\sim N}}N^{-\frac{1}{2}-\alpha}(N^{\frac{1}{3}}+|2X-\nu_n|)^{-\frac{1}{2}}\lesssim_\alpha N^{-\frac{1}{2}-\alpha}(N^{-\frac{1}{6}}+N\cdot N^{-\frac{1}{2}})\\
     &\lesssim_\alpha N^{-\alpha}.
    \end{align*}
    In the remaining exponential regime it is clear that we also have
    \begin{align*}
     \sum_{\substack{n\sim N}}|b_n(X)|^2&\lesssim_\alpha  N^{\max\left\{\alpha+1,-\alpha\right\}}.
    \end{align*}
    This settles the proof in the case $l=0$. We now analyze the cases $l=1$ and $l=2$.
    We recall that
    \begin{align*}
        B_n^{(l)}(X)&=(-1)^le^{-X}\sum_{\substack{j=0}}^l2^j\binom{l}{j}L^{\alpha+j}_{n-j}(2X).
    \end{align*}
    Let $h_{m,\beta}:=2^{-\beta-1}\frac{\Gamma(m+\beta+1)}{m!}$, and $\displaystyle b_{m,\beta}(\cdot)=h_{m,\beta}^{-\frac{1}{2}}B_m(\beta,\cdot)$. We then have
    \begin{align*}
        h_{n,\alpha}^{-\frac{1}{2}}e^{-X}L^{\alpha+j}_{n-j}(2X)=\left(\frac{h_{n-j,\alpha+j}}{h_{n,\alpha}}\right)^{\frac{1}{2}}b^{\alpha+j}_{n-j}(X).
    \end{align*}
    We can also see that
    \begin{align*}
       \frac{h_{n-j,\alpha+j}}{h_{n,\alpha}}=2^{-j}\frac{n!}{(n-j)!} 
    \end{align*}
    When $n\sim N$, we also have
    \begin{align*}
       \frac{h_{n-j,\alpha+j}}{h_{n,\alpha}}\lesssim_j N^j. 
    \end{align*}
    When $j\geq 1$, we have $\alpha+j\geq 0$, and applying the previous case gives the estimate
    \begin{align*}
      |b^{\alpha+j}_{n-j}(X)|^2\lesssim_\alpha N^{\alpha+j}.  
    \end{align*}
    Thus, 
    \begin{align*}
        \sup_{X}\sum_{\substack{n\sim N}}|h_{n,\alpha}^{-\frac{1}{2}}e^{-X}L^{\alpha+j}_{n-j}(2X)|^2\lesssim_\alpha N^{\alpha+1+2j}. 
    \end{align*}
    When $j=0$, we instead use our previous bound
    \begin{align*}
     \sup_{X}\sum_{\substack{n\sim N}}|b_n(X)|^2&\lesssim_\alpha  N^{\max\left\{\alpha+1,-\alpha\right\}}.
    \end{align*}
    Since for $j\geq 1$, $\alpha+1+2j\geq \max\left\{\alpha+1,-\alpha\right\}$, we deduce that
    \begin{align*}
     \sup_{X}\sum_{\substack{n\sim N}}|b^{(l)}_n(X)|^2&\lesssim_\alpha  N^{\alpha+1+2l},
    \end{align*}
    as desired.
    The second estimate follows immediately from the Cauchy-Buniakowski-Schwartz inequality.
\end{proof}
We now record a result which controls the relevant derivatives and weighted derivatives of the gallery waves belonging to our dyadic cluster. We have the following 
\begin{lemma}\label{Uniform Derivative bounds lemma}
    Let $\kappa>0$. We also denote $D:=x\partial_x$. For every integers $l=\overline{0,2}$ and $J\geq 0$, we have
    \begin{align*}
      D^Jb_n=\sum_{\substack{|r|\leq J}}c_{J,r}(n)b_{n+r},  
    \end{align*}
    where $b_n=0$ for $n<0$, and we also have the bound
    \begin{align*}
      |c_{J,r}(n)|&\lesssim_{J,\kappa}(1+n)^J.  
    \end{align*}
    In particular, this implies that
    \begin{align*}
        \sup_{X\geq 0}\sum_{\substack{n\sim N}}|D^Jb_n^{(l)}(X)|^2&\lesssim_{\kappa,J}N^{a_{\kappa,l}+2J}.
    \end{align*}
    Moreover, if $\rho$ ranges in a compact interval, we also have the estimate
    \begin{align*}
        \sup_{X,Y\geq 0}\sum_{\substack{n\sim N}}|(D^J(b_n^{(l)}(\rho X)b_n^{(l)}(\rho Y))|&\lesssim_{\kappa,J}N^{a_{\kappa,l}+J}.
    \end{align*}
    
\end{lemma}
\begin{proof}
    We shall first derive a recurrence relation for the sequence $(B_n)$. We have
    \begin{align*}
        DB_n(X)&=X\partial_X(e^{-X}L_n^\alpha(2X))\\
        &=-XB_n(X)+e^{-X}2X(L_n^\alpha)'(2X)\\
        &=-XB_n(X)+e^{-X}(nL_n^\alpha(2X)-(n+\alpha)L^\alpha_{n-1}(2X))\\
        &=-XB_n(X)+nB_n(X)-(n+\alpha)B_{n-1}(X)
    \end{align*}
    Here we have used the relation
    \begin{align*}
x(L_n^\alpha)'(x)&=nL_n^\alpha(x)-(n+\alpha)L^\alpha_{n-1}(x).
    \end{align*}
    We also note the relation
    \begin{align*}
        xL_n^\alpha(x)&=(2n+\alpha+1)L_n^\alpha(x)-(n+1)L_{n+1}^\alpha(x)-(n+\alpha)L_{n-1}^\alpha(x).
    \end{align*}
    These imply that
    \begin{align*}
        DB_n(X)&=\frac{n+1}{2}B_{n+1}(X)-\frac{\alpha+1}{2}B_{n}(X)-\frac{n+\alpha}{2}B_{n-1}(X).
    \end{align*}
    Thus,
    \begin{align*}
        Db_n(X)&=\frac{n+1}{2}\left(\frac{h_{n+1}}{h_n}\right)^{\frac{1}{2}}b_{n+1}(X)-\frac{\alpha+1}{2}b_{n}(X)-\frac{n+\alpha}{2}\left(\frac{h_{n-1}}{h_n}\right)^{\frac{1}{2}}b_{n-1}(X).
    \end{align*}
    Since
    \begin{align*}
        \frac{h_{n+1}}{h_n}&=\frac{\Gamma(n+\alpha+2)}{(n+1)!}\frac{n!}{\Gamma(n+\alpha+1)}=\frac{n+\alpha+1}{n+1}\\
        \frac{h_{n-1}}{h_n}&=\frac{\Gamma(n+\alpha)}{(n-1)!}\frac{n!}{\Gamma(n+\alpha+1)}=\frac{n}{n+\alpha},
    \end{align*}
    we deduce that
    \begin{align*}
        Db_n(X)&=\frac{\sqrt{(n+1)(n+\alpha+1)}}{2}b_{n+1}(X)-\frac{\alpha+1}{2}b_{n}(X)-\frac{\sqrt{n(n+\alpha)}}{2}b_{n-1}(X).
    \end{align*}
    We can immediately see that when $n\sim N$, all of these coefficients are $O_\kappa(N)$. Upon iterating the relation, we obtain the first two conclusions.

    We now note that for every nonnegative integer $l$ we have $\displaystyle D\partial_X^l=\partial_X^l(D-l)$, hence for every nonnegative integer $J$, $\displaystyle D^J\partial_X^l=\partial_X^l(D-l)^J$. We now immediately get that
    \begin{align*}
D^Jb_n^{(l)}&=\sum_{\substack{j=0}}^J\sum_{\substack{|r|\leq j}}c^{J,l}_{j,r}(n)b^{(l)}_{n+r}.,
    \end{align*}
    where the coefficients $c^{J,l}_{j,r}(n)$ satisfy the estimates $\displaystyle |c^{J,l}_{j,r}(n)|\lesssim (1+n)^J$.

    We now note that when $n\sim N$, the indices $n+r$ remain in an enlarged dyadic region containing $n$. Therefore,
    \begin{align*}
        \sum_{\substack{n\sim N}}|D^Jb_n^{(l)}(X)|^2\lesssim_{\kappa,J} N^{2J}\sum_{\substack{n\sim C_JN}}|b_n^{(l)}(X)|^2\lesssim_{\kappa,J} N^{a_{\kappa,l}+2J}, 
    \end{align*}
    where we have also used Lemma \ref{Uniform Lemma} for the second estimate.

    For the last estimate, we can immediately see that
    \begin{align*}
        (\rho\partial_\rho)^J(b_n^{(l)}(\rho X)b_n^{(l)}(\rho Y))&=\sum_{\substack{a+b=J}}\binom{J}{a}(D^ab_n^{(l)})(\rho X)(D^bb_n^{(l)})(\rho Y)
    \end{align*}
    If we fix a pair of indices $(a,b)$ with $a+b=J$, Cauchy-Buniakowski-Schwartz implies that
    \begin{align*}
        \sum_{\substack{n\sim N}}|D^ab_n^{(l)}(\rho X)D^bb_n^{(l)}(\rho Y)|&\lesssim \left(\sum_{\substack{n\sim N}}|D^ab_n^{(l)}(\rho X)|^2\right)^{\frac{1}{2}}\left(\sum_{\substack{n\sim N}}|D^bb_n^{(l)}(\rho Y)|^2\right)^{\frac{1}{2}}\\
        &\lesssim N^{\frac{a_{\kappa,l}}{2}+a}N^{\frac{a_{\kappa,l}}{2}+b}=N^{a_{\kappa,l}+J}.
    \end{align*}
    Once we sum over all pairs of indices $(a,b)$ with $a+b=J$, we deduce the desired estimate.
\end{proof}
In what follows, we prove some estimates for the $TT^\ast$ kernel which shall be essential in our argument. We have the following
\begin{lemma}\label{Kernel bounds lemma}
    Let $\kappa>0$, $\psi\in C_c^\infty(\mathbb{R}^{d-1}-\{0\})$ be a radial function, and $a\in C^\infty(\mathcal{S}^{d-2})$. We also assume that there exist some positive constants $c_0$ and $C_0$ such that for every $\eta\in\supp\psi$, $c_0\leq |\eta|\leq C_0$. For every integer $l=\overline{0,2}$, we define the cluster kernel
    \begin{align*}
        K_{N,l}(\sigma,z,X,Y)=\sum_{\substack{n\sim N}}\int_{\mathbb{R}^{d-1}}e^{i\left(z\cdot \eta+\sigma\left(\frac{\mu_n}{N}\right)^{\frac{1}{2}}|\eta|^{\frac{1}{2}}\right)}b_n^{(l)}(|\eta|X)b_n^{(l)}(|\eta|Y)\psi^2(\eta)a\left(\frac{\eta}{|\eta|}\right)\,d\eta, 
    \end{align*}
    where $\sigma,X,Y\geq 0$, and $z\in\mathbb{R}^{d-1}$. Then, we have the following:
    \begin{enumerate}
        \item[i.] For every $\sigma$, $z$, $X$, and $Y$,
        \begin{equation}\label{Uniform estimate}
           |K_{N,l}(\sigma,z,X,Y)|\lesssim N^{a_{\kappa,l}} 
        \end{equation}
        \item[ii.] There exist positive constants $c$ and $C$ depending only on $\kappa$ and $\supp\psi$ such that whenever $c\sigma\leq|z|\leq C\sigma$,
        \begin{equation}\label{Stationary estimate}
           |K_{N,l}(\sigma,z,X,Y)|\lesssim N^{1+a_{\kappa,l}}\sigma^{-\frac{d-1}{2}}
        \end{equation}
        \item[iii.] If, on the other hand, $|z|\leq c\sigma$ or $|z|\geq C\sigma$, then for every integer $J\geq 1$,
        \begin{equation}\label{Non-stationary estimate}
           |K_{N,l}(\sigma,z,X,Y)|\lesssim N^{J+a_{\kappa,l}}(\sigma+|z|)^{-J}.
        \end{equation}
    \end{enumerate}
\end{lemma}
\begin{proof}
    From Lemma \ref{Uniform Lemma}, we have
    \begin{align*}
        \sup_{w\geq 0}\sum_{\substack{n\sim N}}|b_n^{(l)}(w)|^2\lesssim N^{a_{\kappa,l}}.
    \end{align*}
    The Cauchy-Buniakowski-Schwartz inequality immediately implies that for every $X,Y\geq 0$,
    \begin{align*}
        \sum_{\substack{n\sim N}}|b_n^{(l)}(X)||b_n^{(l)}(Y)|&\lesssim \left(\sum_{\substack{n\sim N}}|b_n^{(l)}(X)|^2\right)^{\frac{1}{2}}\left(\sum_{\substack{n\sim N}}|b_n^{(l)}(Y)|^2\right)^{\frac{1}{2}}\\
        &\lesssim N^{a_{\kappa,l}}
    \end{align*}
    Now, since $\psi$ has compact support, we immediately get that
     \begin{align*}
        |K_{N,l}(\sigma,z,X,Y)|&\lesssim N^{a_{\kappa,l}}.
    \end{align*}

    We now move on to the second bound, which corresponds to the stationary regime. In what follows, we assume that $\sigma\geq 1$, since for $\sigma<1$, the desired bound immediately follows from Lemma \ref{Uniform estimate}. Let $\eta=\rho\omega$, where $\rho=|\eta|$, and $\omega\in\mathbb{S}^{d-2}$.

    We first analyze the case $d\geq 3$. We have
    \begin{align*}
        K_{N,l}(\sigma,z,X,Y)&=\sum_{\substack{n\sim N}}\int_{\mathbb{R}^{d-1}}e^{i\left(z\cdot \eta+\sigma\left(\frac{\mu_n}{N}\right)^{\frac{1}{2}}|\eta|^{\frac{1}{2}}\right)}b_n^{(l)}(|\eta|X)b_n^{(l)}(|\eta|Y)\psi^2(\eta)a\left(\frac{\eta}{|\eta|}\right)\,d\eta\\
        &=C_d\sum_{\substack{n\sim N}}\int_{\mathbb{R}}e^{i\sigma\left(\frac{\mu_n}{N}\right)^{\frac{1}{2}}\rho^{\frac{1}{2}}}b_n^{(l)}(\rho X)b_n^{(l)}(\rho Y)\psi^2(\rho)\rho^{d-2}\,d\rho\int_{\mathbb{S}^{d-2}}e^{i\rho z\cdot\omega}a(\omega)\,d\omega. 
    \end{align*}
    Let $\mathcal{S}(\rho z):=\int_{\mathbb{S}^{d-2}}e^{i\rho z\cdot\omega}a(\omega)\,d\omega$.

    On the sphere $\mathbb{S}^{d-2}$, the phase inside the integral of $\mathcal{S}(\rho z)$ satisfies the relation 
    \begin{align*}
    \nabla_{\mathbb{S}^{d-2}}(\rho z\cdot\omega)=\rho\left(\frac{z}{|z|}-\left(\frac{z}{|z|}\cdot\omega\right)\cdot\omega\right),
    \end{align*}
    hence it has the two critical points $\pm\frac{z}{|z|}$. By Lemma \ref{Stationary phase}, it follows that
    \begin{align*}
      \mathcal{S}(\rho z)&=|z|^{-\frac{d-2}{2}}\left(e^{i\rho|z|}a_+\left(\rho,z\right)+e^{-i\rho|z|}a_-\left(\rho,z\right)\right)+\mathcal{R}(\rho,z),
    \end{align*}
    where
    \begin{align*}
        |\partial_\rho^j\mathcal{R}(\rho,z)|\lesssim_{j,M}|z|^{-M}, 
    \end{align*}
    and $a_\pm(\rho,z)$ are smooth amplitudes satisfying $|\partial_\rho^ja_\pm(\rho,z)|\lesssim_{j,M} 1$, whose leading order expansions are 
    \begin{align*}
a_+(\rho,z)&:=\rho^{-\frac{d-2}{2}}(2\pi)^{\frac{d-2}{2}}e^{-\frac{i\pi(d-2)}{4}}a\left(\frac{z}{|z|}\right)+a^{(1)}_+(\rho,z)\\
a_-(\rho,z)&:=\rho^{-\frac{d-2}{2}}(2\pi)^{\frac{d-2}{2}}e^{\frac{i\pi(d-2)}{4}}a\left(-\frac{z}{|z|}\right)+a^{(1)}_-(\rho,z).
    \end{align*}
    
We also note that for every nonnegative integer $j$, $\displaystyle |\partial_\rho^ja_\pm(\rho,z)|\leq C_j$. The contribution of $\mathcal{R}(\rho,z)$ can be immediately seen to be favorable for large enough $M$, since estimate \eqref{Uniform estimate} immediately implies that
\begin{align*}
    \left|C_d\sum_{\substack{n\sim N}}\int_{\mathbb{R}}e^{i\sigma\left(\frac{\mu_n}{N}\right)^{\frac{1}{2}}\rho^{\frac{1}{2}}}b_n^{(l)}(\rho X)b_n^{(l)}(\rho Y)\psi^2(\rho)\rho^{d-2}\mathcal{R}(\rho,z)\,d\rho\right|&\lesssim_{M} N^{a_{\kappa,l}}|z|^{-M}.
\end{align*}
We are left to consider the stationary contribution in the expression of $K_{N,l}$:
\begin{align*}
    C_d|z|^{-\frac{d-2}{2}}\sum_{\substack{n\sim N}}\int_{\mathbb{R}}e^{i\sigma\left(\frac{\mu_n}{N}\right)^{\frac{1}{2}}\rho^{\frac{1}{2}}\pm i\rho|z|}b_n^{(l)}(\rho X)b_n^{(l)}(\rho Y)a_\pm(\rho,z)\psi^2(\rho)\rho^{d-2}\,d\rho 
\end{align*}
We note that the second derivatives of the $\rho$-phases $\displaystyle \sigma\left(\frac{\mu_n}{N}\right)^{\frac{1}{2}}\rho^{\frac{1}{2}}\pm \rho|z|$ are both equal to $\displaystyle-\frac{1}{4} \sigma\left(\frac{\mu_n}{N}\right)^{\frac{1}{2}}\rho^{-\frac{3}{2}}$. Since $n\sim N$, $\mu_n=1+2n\kappa$, and $\rho\in\supp\psi$, it follows that these second derivatives are uniformly comparable to $\sigma$.

Let $
\displaystyle A^\pm_{n,l}(\rho,z):=b_n^{(l)}(\rho X)b_n^{(l)}(\rho Y)\rho^{m-1}\psi^2(\rho)a_\pm(\rho,z)$. By the van der Corput Lemma, 
\begin{align*}
  \left|\int_{\mathbb{R}}e^{i\sigma\left(\frac{\mu_n}{N}\right)^{\frac{1}{2}}\rho^{\frac{1}{2}}\pm i\rho|z|}a_\pm(\rho,z)b_n^{(l)}(\rho X)b_n^{(l)}(\rho Y)\psi^2(\rho)\rho^{d-2}\,d\rho\right|&\lesssim \sigma^{-\frac{1}{2}}(\|A^\pm_{n,l}(\rho,z)\|_{L_\rho^\infty}+\|\partial_\rho A^\pm_{n,l}(\rho,z)\|_{L_\rho^1})  
\end{align*}
By Lemmas \ref{Uniform Lemma} and \ref{Uniform Derivative bounds lemma},and the comact support condition, we have
\begin{align*}
    \sum_{\substack{n\sim N}}\|A^\pm_{n,l}(\rho,z)\|_{L_\rho^\infty}&\lesssim N^{a_{\kappa,l}}\\
    \sum_{\substack{n\sim N}}\|\partial_\rho A^\pm_{n,l}(\rho,z)\|_{L_\rho^1}&\lesssim N^{1+a_{\kappa,l}}. 
\end{align*}
It immediately follows that 
\begin{align*}
  | K_{N,l}(\sigma,z,X,Y)|&\lesssim |z|^{-\frac{d-2}{2}}\sigma^{-\frac{1}{2}}N^{1+a_{\kappa,l}}. 
\end{align*}
Since $c_0\sigma\leq|z|\leq C_0\sigma$, we deduce that
\begin{align*}
  | K_{N,l}(\sigma,z,X,Y)|&\lesssim \sigma^{-\frac{d-1}{2}}N^{1+a_{\kappa,l}}, 
\end{align*}
as desired. The case $d=2$ is much simpler, since the sphere only consists of the points $-1$ and $1$, and the integral becomes a sum which can be once again directly estimated using the same argument involving van der Corput's lemma as before.

We now move on to the third estimate. We note that the first derivatives of the $\rho$-phases $\displaystyle \sigma\left(\frac{\mu_n}{N}\right)^{\frac{1}{2}}\rho^{\frac{1}{2}}\pm \rho|z|$ are equal to $\displaystyle\frac{1}{2} \sigma\left(\frac{\mu_n}{N}\right)^{\frac{1}{2}}\rho^{-\frac{1}{2}}\pm|z|$. Since $n\sim N$, $\mu_n=1+2n\kappa$, and $\rho\in\supp\psi$, we can see that the two constants $c_0$ and $C_0$ could have been chosen so that whenever $|z|\leq c_0\sigma$ or $|z|\geq C_0\sigma$, the two phases of the derivatives are both bounded by $\sigma+|z|$ from below.

We first analyze the case $|z|\geq C\sigma$. We perform the same decomposition as in the proof of the second estimate. With the previous notations, we have
\begin{align*}
    K_{N,l}(\sigma,z,X,Y)&=C_d|z|^{-\frac{d-2}{2}}\sum_{\substack{n\sim N}}\int_{\mathbb{R}}e^{i\sigma\left(\frac{\mu_n}{N}\right)^{\frac{1}{2}}\rho^{\frac{1}{2}}\pm i\rho|z|}b_n^{(l)}(\rho X)b_n^{(l)}(\rho Y)a_\pm(\rho,z)\psi^2(\rho)\rho^{d-2}\,d\rho\\
    &+C_d\sum_{\substack{n\sim N}}\int_{\mathbb{R}}e^{i\sigma\left(\frac{\mu_n}{N}\right)^{\frac{1}{2}}\rho^{\frac{1}{2}}}b_n^{(l)}(\rho X)b_n^{(l)}(\rho Y)\psi^2(\rho)\rho^{d-2}\mathcal{R}(\rho,z)\,d\rho.
\end{align*}
Just as before, we have the bound
\begin{align*}
    \left|C_d\sum_{\substack{n\sim N}}\int_{\mathbb{R}}e^{i\sigma\left(\frac{\mu_n}{N}\right)^{\frac{1}{2}}\rho^{\frac{1}{2}}}b_n^{(l)}(\rho X)b_n^{(l)}(\rho Y)\psi^2(\rho)\rho^{d-2}\mathcal{R}(\rho,z)\,d\rho\right|&\lesssim_{M} N^{a_{\kappa,l}}|z|^{-M},
\end{align*}
which for large enough $M$, is favorable. We are once again left to analyze
\begin{align*}
    C_d|z|^{-\frac{d-2}{2}}\sum_{\substack{n\sim N}}\int_{\mathbb{R}}e^{i\sigma\left(\frac{\mu_n}{N}\right)^{\frac{1}{2}}\rho^{\frac{1}{2}}\pm i\rho|z|}A^\pm_{n,l}(\rho,z)\,d\rho
\end{align*}
We shall integrate by parts $J$ times using the operators $\displaystyle \frac{1}{i\left(\frac{1}{2} \sigma\left(\frac{\mu_n}{N}\right)^{\frac{1}{2}}\rho^{-\frac{1}{2}}\pm|z|\right)}\partial_\rho$. We end up with 
 \begin{align*}
     C_d|z|^{-\frac{d-2}{2}}\sum_{\substack{n\sim N}}\int_{\mathbb{R}}e^{i\sigma\left(\frac{\mu_n}{N}\right)^{\frac{1}{2}}\rho^{\frac{1}{2}}\pm i\rho|z|}\left(\frac{1}{i\left(\frac{1}{2} \sigma\left(\frac{\mu_n}{N}\right)^{\frac{1}{2}}\rho^{-\frac{1}{2}}\pm|z|\right)}\partial_\rho\right)^JA^\pm_{n,l}(\rho,z)\,d\rho.
 \end{align*}
 We note the bound
 \begin{align*}
     \left|\left(\frac{1}{i\left(\frac{1}{2} \sigma\left(\frac{\mu_n}{N}\right)^{\frac{1}{2}}\rho^{-\frac{1}{2}}\pm|z|\right)}\partial_\rho\right)^JA^\pm_{n,l}(\rho,z)\right|&\lesssim_J (\sigma+|z|)^{-J}\sum_{\substack{m=0}}^J|\partial_\rho^mA^\pm_{n,l}(\rho,z)|.
 \end{align*}
 By Lemma 6.2, we have the estimate
 \begin{align*}
    \sum_{\substack{n\simeq N}}\sum_{\substack{m=0}}^J|\partial_\rho^mA^\pm_{n,l}(\rho,z)|&\lesssim_{J,\kappa}N^{J+a_{\kappa,l}}.
 \end{align*}
Thus,
\begin{align*}
  |K_{N,l}(\sigma,z,X,Y)|&\lesssim_{J,\kappa}(\sigma+|z|)^{-J}N^{J+a_{\kappa,l}}  
\end{align*}
We now analyze the case $|z|\leq c\sigma$. We have 
\begin{align*}
    \nabla_\eta\left(z\cdot \eta+\sigma\left(\frac{\mu_n}{N}\right)^{\frac{1}{2}}|\eta|^{\frac{1}{2}}\right)=z+\sigma\left(\frac{\mu_n}{N}\right)^{\frac{1}{2}}\frac{\eta}{|\eta|^{\frac{3}{2}}}.
\end{align*}
It is not difficult to check that we may choose $c_0$ sufficiently small from the very beginning so that
\begin{align*}
    \left|\nabla_\eta\left(z\cdot \eta+\sigma\left(\frac{\mu_n}{N}\right)^{\frac{1}{2}}|\eta|^{\frac{1}{2}}\right)\right|\geq c\sigma,
\end{align*}
for where $c$ is a positive constant uniform in $n$, $N$, and $\eta\in\supp\psi$. We also note that for $J\geq 2$, 
\begin{align*}
    \left|\nabla_\eta^J\left(z\cdot \eta+\sigma\left(\frac{\mu_n}{N}\right)^{\frac{1}{2}}|\eta|^{\frac{1}{2}}\right)\right|\geq c_J\sigma,
\end{align*}
since $z\cdot\eta$ vanishes. Let $\displaystyle V_n(\eta):=\frac{\nabla_\eta\left(z\cdot \eta+\sigma\left(\frac{\mu_n}{N}\right)^{\frac{1}{2}}|\eta|^{\frac{1}{2}}\right)}{i\left|\nabla_\eta\left(z\cdot \eta+\sigma\left(\frac{\mu_n}{N}\right)^{\frac{1}{2}}|\eta|^{\frac{1}{2}}\right)\right|^2}$. In particular, we get that $|\nabla_\eta^JV_n(\eta)|\lesssim_J\sigma^{-1}$ uniformly on $\supp\psi$.

We shall integrate by parts in the expression of $K_{N,l}$ $J$ times using the operator 
\begin{align*}
    V_n(\eta)\cdot \nabla_\eta.
\end{align*}
If $T_n$ denotes its adjoint, we deduce that
\begin{align*}
        K_{N,l}(\sigma,z,X,Y)&=\sum_{\substack{n\sim N}}\int_{\mathbb{R}^{d-1}}e^{i\left(z\cdot \eta+\sigma\left(\frac{\mu_n}{N}\right)^{\frac{1}{2}}|\eta|^{\frac{1}{2}}\right)}T_n^JA_{n,l}(\eta)\,d\eta, 
    \end{align*}
    where $A_{n,l}(\eta)=b_n^{(l)}(|\eta|X)b_n^{(l)}(|\eta|Y)\psi^2(\eta)$. We note that
    \begin{align*}
       | T_n^JA_{n,l}(\eta)|\lesssim_J\sigma^{-J}\sum_{\substack{m=0}}^J|\partial^m_\eta A_{n,l}(\eta)|
    \end{align*}
    We note that there exists uniformly bounded angular coefficients $c_{\beta,J}$ such that for every radial function $f(\rho)$,
    \begin{align*}
       \partial^\gamma_\eta f(\eta)=\sum_{|\beta|\leq |\gamma|}c_{\beta,|\gamma|}(\eta)(\rho\partial_\rho)^{|\beta|}f(\rho). 
    \end{align*}
    By Lemma \ref{Uniform Derivative bounds lemma}, when $f(\rho)=b_n^{(l)}(\rho X)$ or $f(\rho)=b_n^{(l)}(\rho Y)$, we have
    \begin{align*}
        \sum_{\substack{n\sim N}}|(\rho\partial_\rho)^mf(\rho)|\lesssim N^{a_{\kappa,l}+m},
    \end{align*}
    hence
    \begin{align*}
        \sum_{\substack{n\sim N}}\sum_{\substack{m=0}}^J|\partial^m_\eta A_{n,l}(\eta)|&\lesssim_JN^{a_{\kappa,l}+J}.
    \end{align*}
    Finally, this implies the estimate
    \begin{align*}
        |K_{N,l}(\sigma,z,X,Y)|\lesssim \sigma^{-J}N^{a_{\kappa,l}+J}\lesssim (\sigma+|z|)^{-J}N^{a_{\kappa,l}+J},
    \end{align*}
    as desired.
\end{proof}
Let $(u_t,\nabla_xu)\in \mathcal{H}^{2s}$, and $
    \displaystyle u=\sum_{\substack{N\in 2^{\mathbb{Z}_{\geq 0}}}}\sum_{\substack{n\sim N}}u_n$ be its decomposition into gallery modes. We recall that $\displaystyle \tilde{L}u_n=\mu_n|\nabla_{x'}|u_n$. We further decompose $u_n=\sum_{\substack{j}}u_{n,j}$, where $u_{n,j}=P_{x',2^{2j}}u_n$ is the Littlewood-Paley component of $u_n$ at tangential frequency $2^{2j}$. Let $E_{N,j}:=\|\partial_t u_{N,j},\nabla u_{N,j}\|^2_{\mathcal{H}}$. For these dyadic energy pieces, we have the following 
\begin{lemma}\label{Energy dyadic blocks lemma}
    In the previous notations, we have
    \begin{align*}
       \sum_{\substack{j\geq 0, N\geq 1}}(N2^{2j})^{2k}E_{N,j}&\lesssim \|(u_t,\nabla u)\|^2_{\mathcal{H}^{2k}}.
    \end{align*}
\end{lemma}
\begin{proof}
   We shall first prove the estimate when $k$ is an integer. From the definition of the spaces $\mathcal{H}^{2k}$ and the weighted elliptic estimates from Ifrim-Tataru \cite{Ifrim-Tataru} (see sections 2.2 and Theorem 7 in section 7 in \cite{Ifrim-Tataru}), we know that 
   \begin{align*}
       \sum_{\substack{m=0}}^k\|(\partial_t(\tilde{L})^mu,\nabla (\tilde{L})^mu)\|^2_{\mathcal{H}}&\lesssim \|(u_t,\nabla u)\|^2_{\mathcal{H}^{2k}}.
   \end{align*}
   We apply these for $u_{N,j}$. Since for every $n\sim N$, $\tilde{L}u_n=\mu_n|\nabla_{x'}|u_n$, it immediately follows that $\tilde{L}P_{x',2^{2j}}u_n=\mu_n|\nabla_{x'}|P_{x',2^{2j}}u_n$, hence
   \begin{align*}
       \sum_{\substack{j\geq 0, N\geq 1}}(N2^{2j})^{2k}E_{N,j}&\lesssim 
    \sum_{\substack{m=0}}^k\|(\partial_t(\tilde{L})^mu,\nabla (\tilde{L})^mu)\|^2_{\mathcal{H}}\lesssim \|(u_t,\nabla u)\|^2_{\mathcal{H}^{2k}},
   \end{align*}
   as desired.

   For noninteger $k$, the desired result follows by interpolation.
\end{proof}
We now return to the proof of Theorem \ref{Strichartz Estimates}.
\begin{proof}[Proof of Theorem \ref{Strichartz Estimates}]
    It suffices to prove the endpoint $L_t^qL_x^\infty$ estimate, since the rest follow by interpolation with the energy case, which can be immediately proved using the Sobolev estimates and the energy estimates from Lemmas 2.3-2.4 in \cite{Ifrim-Tataru}.

    We first note that our conditions on the regularity index $s$ require us to only choose the Laguerre branch for our gallery waves, including in the case $\kappa>1$. This because on the unwanted Frobenius branch, we have $f(x_d)\sim x_d^{1-\frac{1}{\kappa}}$, hence $x_d^{\frac{1}{2\kappa}}\partial_d^2f\sim x_d^{-1-\frac{1}{2\kappa}}$, which is never square integrable around $0$. This would in particular rule out the cases where the initial data belongs to $\mathcal{H}^{2s}$, with $s\geq 1$, which is incompatible with our assumptions.
    
    We shall first analyze the case $q(d-1)>4$, which always occurs, unless $d=2$ $q=4$, and $r=\infty$. Let $\displaystyle s_c(q):=\frac{d}{2}+\frac{1}{2\kappa}+1-\frac{1}{2q}$. In what follows, we denote by $\psi$ the radial cut-off defining the tangential Littlewood-Paley projectors. 

    We first note that for any given tangential frequency $\xi'\neq 0$, the functions given by $\displaystyle b_n(|\xi'|x_d)=h_n^{-\frac{1}{2}}e^{-|\xi'|x_d}L_n^\alpha(2|\xi'|x_d)$, with $\alpha=\frac{1}{\kappa}-1$, form a complete orthogonal system in $L^2((0,\infty), x_d^{\frac{1}{\kappa}-1}\,dx_d)$. Indeed this immediately follows from the classical completeness of the generalized Laguerre polynomials $L_n^\alpha$ in $L^2((0,\infty), e^{-\frac{y}{2}}y^{\frac{1}{\kappa}-1}\,dy)$ upon the change of variables $y=2|\xi'|x_d$. In particular, this means that after we take the tangential Fourier transform, every element of our relevant weighted $L^2$ Sobolev space admits an expansion into gallery modes. Moreover, the additional regularity that we have imposed in our hypotheses for Theorem \ref{Strichartz Estimates} guarantees that the non-Laguerre Frobenius branch is excluded in the case $\kappa>1$ as well (when $\kappa\leq 1$ it had naturally been excluded by our energy class), this decomposition is compatible with the realization of the operator $\tilde{L}$ that we discussed in Section 4.

    By generalized Laguerre completeness, we now decompose the initial data $(u_0,u_1)$ using the gallery modes as
    \begin{align*}
        \widehat{u}_0(\xi',x_d)&=\sum_{\substack{n}}b_n(|\xi'|x_d)\widehat{f}_n(\xi')\\
        \widehat{u}_1(\xi',x_d)&=\sum_{\substack{n}}b_n(|\xi'|x_d)\widehat{g}_n(\xi')
    \end{align*}
    We also define the half-wave coefficients
    \begin{align*}
        \widehat{F}^\pm_n(\xi')&=\frac{1}{2}(\widehat{f}_n(\xi')\mp\frac{i}{\sqrt{\mu_n|\xi'|}}\widehat{g}_n(\xi')).
    \end{align*}
    We can immediately check that we have the following relations
    \begin{align*}
\widehat{f}_n(\xi')&=\widehat{F}^+_n(\xi')+\widehat{F}^-_n(\xi')\\
        \widehat{g}_n(\xi')&=i\sqrt{\mu_n|\xi'|}(\widehat{F}^+_n(\xi')-\widehat{F}^-_n(\xi')).
    \end{align*}
    We now define $\displaystyle F^\pm=(F_n^\pm)_{n\geq 0}$. For a joint dyadic block corresponding to $N$, we also define $F^\pm_{N,j}=(P_{x',2^{2j}}F^\pm_{n})_{n\sim N}$.

    Given an arbitrary sequence $\displaystyle (G_n)_{n\sim N}$, we also introduce the propagators $\mathcal{U}^{\pm}_{N,j}(t)$ defined by
    \begin{align*}
       \mathcal{U}^{\pm}_{N,j}(t)G(x',x_d)&=\sum_{
       \substack{n\sim N}}\int_{\mathbb{R}^{d-1}}e^{ix'\cdot\xi'}e^{\pm it\sqrt{\mu_n|\xi'|}}b_n(|\xi'|x_d)\psi\left(\frac{\xi'}{2^{2j}}\right)\widehat{G}_n(\xi')d\,\xi'.
    \end{align*}
    
Let $u$ be a solution. We decompose it with respect to the tangential frequency and the the gallery modes as
    \begin{align*}
        u=\sum_{\substack{n,j}}u_{n,j}=\sum_{\substack{N,j}}u_{N,j}.
    \end{align*}
    Here, $\displaystyle u_{N,j}=\sum_{\substack{n\sim N}}u_{n,j}$, and in each $u_{n,j}$ component, the tangential frequency will satisfy $|\xi'|\simeq 2^{2j}$. To be more precise, we have $u_{N,j}^\pm(t)=\mathcal{U}^\pm_{N,j}(t)F^{\pm}_{N,j}$.
    
    In what follows, we shall estimate each $u_{N,j}$ term. We know that the latter can be written as a sum of half-waves $u_{N,j}=u^+_{N,j}+u^-_{N,j}$, where
    \begin{align*}
        u^{\pm}_{N,j}&=\sum_{\substack{n\sim N}}\int_{\mathbb{R}^{d-1}}e^{ix'x\dot\xi'\pm it\sqrt{\mu_n|\xi'|}}b_n(|\xi'|x_d)\psi\left(\frac{\xi'}{2^{2j}}\right)\widehat{F_n}^{\pm}(\xi')\,d\xi'. 
    \end{align*}
We shall only estimate $u_{N,j}^+$, since the other estimate is similar.
    
    For an integer $l=\overline{0,2}$ and a multi-index $\beta$ with $|\beta|=2-l$, we define $T_{l,\beta}:=\partial_{x_d}^l\partial^{\beta}_{x'}$, and it suffices for us to prove the desired estimate for each $T_{\beta,l}u$. We also denote $V_{N,j}^{l,\beta}(t):=T_{l,\beta}\mathcal{U}^+_{N,j}(t)$.

    We have
    \begin{align*}
        V_{N,j}^{l,\beta}(t)G(x)&=\sum_{
       \substack{n\sim N}}\int_{\mathbb{R}^{d-1}}e^{ix'\cdot\xi'}(i\xi)^{\beta}e^{\pm it\sqrt{\mu_n|\xi'|}}|\xi'|^{l}b^{(l)}_n(|\xi'|x_d)\psi\left(\frac{\xi'}{2^{2j}}\right)\widehat{G}_n(\xi')d\,\xi'\\
       &=\sum_{
       \substack{n\sim N}}\int_{\mathbb{R}^{d-1}}e^{ix'\cdot\xi'}i^{\beta}a_{\beta}(\omega)e^{\pm it\sqrt{\mu_n|\xi'|}}|\xi'|^{2}b^{(l)}_n(|\xi'|x_d)\psi\left(\frac{\xi'}{2^{2j}}\right)\widehat{G}_n(\xi')d\,\xi',
    \end{align*}
    where $\omega=\frac{\xi'}{|\xi'|}$ and $a_\beta(\omega)=\omega^\beta$ is an angular function that is smooth on $\mathbb{S}^{d-2}$. 
    We now consider the adjoint operator $(V_{N,j}^{l,\beta}(t))^\ast$,which sends a function to $G$ to a sequence $(G_n)_{n\sim N}$. Modulo normalization constants, we have
    \begin{align*}
        \widehat{(V_{N,j}^{l,\beta}(t))^\ast G_n}(\xi')&=e^{-it\sqrt{\mu_n|\xi'|}}(-i\xi')^{\beta}|\xi'|^l\psi\left(\frac{\xi'}{2^{2j}}\right)\int_0^\infty\int_{\mathbb{R}^{d-1}}e^{-iy'\cdot\xi'}b_n^{(l)}(|\xi'|y_d)G(y',y_d)\,dy'\,dy_d.
    \end{align*}
    We now get that the kernel of $V_{N,j}^{l,\beta}(t)(V_{N,j}^{l,\beta}(s))^\ast$ is given by
    \begin{align*}
        K^{l,\beta}_{N,j}(t-s,x,y)&=\sum_{\substack{n\sim N}}\int_{\mathbb{R}^{d-1}}e^{i\xi'\cdot(x'-y')}e^{i\sqrt{\mu_n|\xi'|}(t-s)}b^{(l)}_n(|\xi'|x_d)b^{(l)}_n(|\xi'|y_d)|a_{\beta}(\omega)|^2|\xi'|^{4}\psi^2\left(\frac{\xi'}{2^{2j}}\right)d\,\xi' ,
    \end{align*}
    with
    \begin{align*}
       V_{N,j}^{l,\beta}(t)(V_{N,j}^{l,\beta}(s))^\ast G(s,x)&=\int_{\mathbb{R}}\int_0^\infty\int_{\mathbb{R}^{d-1}}K^{l,\beta}_{N,j}(t-s,x,y)G(s,y',y_d)\,dy'\,dy_d\,ds.
    \end{align*}
    Let $\xi'=2^{2j}\eta$, $X=2^{2j}x_d$, $Y=2^{2j}y_d$, $z=2^{2j}(x'-y')$, and $\sigma=2^{j}(t-s)\sqrt{N}$.
    
    Since $\displaystyle K^{l,\beta}_{N,j}(t-s,x,y)=\overline{K^{l,\beta}_{N,j}(s-t,y,x)}$, it suffices to analyze the case $t\geq s$, which corresponds to $\sigma\geq 0$. Under these circumstances, we have
    \begin{align*}
        K^{l,\beta}_{N,j}(t-s,x,y)&=2^{2j(d+3)}K^\beta_{N,l}(\sigma,z,X,Y),
    \end{align*}
    where
    \begin{align*}
      K^\beta_{N,l}(\sigma,z,X,Y)&= \sum_{\substack{n\sim N}}\int_{\mathbb{R}^{d-1}}e^{i\left(z\cdot \eta+\sigma\left(\frac{\mu_n}{N}\right)^{\frac{1}{2}}|\eta|^{\frac{1}{2}}\right)}b_n^{(l)}(|\eta|X)b_n^{(l)}(|\eta|Y)|\eta|^4|a_\beta(\omega)|^2\psi^2(\eta)\,d\eta. 
    \end{align*}
    We shall now apply Lemma \ref{Kernel bounds lemma}. When $|z|\simeq \sigma$, we have
    \begin{align*}
         |K^\beta_{N,l}(\sigma,z,X,Y)|&\lesssim N^{1+a_{\kappa,l}}\sigma^{-\frac{d-1}{2}}.
    \end{align*}
    When $|z|\not\simeq \sigma$, for every positive integer $J\geq 1$, we have
    \begin{align*}
         |K^\beta_{N,l}(\sigma,z,X,Y)|&\lesssim N^{J+a_{\kappa,l}}\sigma^{-J}.
    \end{align*}
    We also have the overall uniform bound
    \begin{align*}
         |K^\beta_{N,l}(\sigma,z,X,Y)|&\lesssim N^{a_{\kappa,l}}.
    \end{align*}
    We shall choose $J$ so that $Jq>2$. When $|z|\not\simeq \sigma$, we have the bound
    \begin{align*}
         |K^\beta_{N,l}(\sigma,z,X,Y)|&\lesssim N^{a_{\kappa,l}}\min\left\{1,\left(\frac{N}{\sigma}\right)^J\right\},
    \end{align*}
    while for $|z|\simeq \sigma$, we have
    \begin{align*}
         |K^\beta_{N,l}(\sigma,z,X,Y)|&\lesssim N^{a_{\kappa,l}}\min\left\{1,\frac{N}{\sigma^{\frac{d-1}{2}}}\right\},
    \end{align*}
    hence the $L_\sigma^{\frac{q}{2}}$ norm of the kernel satisfies
    \begin{align*}
         \|K^\beta_{N,l}(\sigma,z,X,Y)\|_{L_\sigma^{\frac{q}{2}}}&\lesssim N^{a_{\kappa,l}+\max\left\{\frac{2}{q},\frac{4}{(d-1)q}\right\}}.
    \end{align*}
    Thus,
    \begin{align*}
         \sup_{\substack{x,y}}\|K^{l,\beta}_{N,j}(\cdot,x,y)\|_{L_\sigma^{\frac{q}{2}}}&\lesssim 2^{2j(d+3)}N^{a_{\kappa,l}+\max\left\{\frac{2}{q},\frac{4}{(d-1)q}\right\}}.
    \end{align*}
    The operator $V_{N,j}^{l,\beta}(t)$ is defined on $l^2_{n\sim N}L^2_{x'}$ with values in  $L_t^qL_x^\infty$, hence $V_{N,j}^{l,\beta}(t)(V_{N,j}^{l,\beta}(s))^\ast$ is defined on $L_t^{q'}L_x^1$ with values in $L_t^qL_x^\infty$. Thus, for arbitrary $G\in L_t^{q'}L_x^1$, we have
    \begin{align*}
        \|(V_{N,j}^{l,\beta}(\cdot))^\ast G\|^2_{l^2_{n\sim N}L^2_{x'}}&=\langle V_{N,j}^{l,\beta}(\cdot)(V_{N,j}^{l,\beta}(s))^\ast G,G\rangle\lesssim \|V_{N,j}^{l,\beta}(\cdot)(V_{N,j}^{l,\beta}(s))^\ast G\|_{L_t^qL_x^\infty}\|G\|_{L_t^{q'}L_x^1}
    \end{align*}
    By Young's inequality if $\tau:=2^{-j}N^{-\frac{1}{2}}\sigma$, we get that
    \begin{align*}
       \|V_{N,j}^{l,\beta}(\cdot)(V_{N,j}^{l,\beta}(s))^\ast G\|_{L_t^qL_x^\infty}&\lesssim   \sup_{\substack{x,y}}\|K^{l,\beta}_{N,j}(\cdot,x,y)\|_{L_\tau^{\frac{q}{2}}}\|G\|_{L_t^{q'}L_x^1}\lesssim   \sup_{\substack{x,y}}\|K^{l,\beta}_{N,j}(\cdot,x,y)\|_{L_\sigma^{\frac{q}{2}}}\|G\|_{L_t^{q'}L_x^1}\\
       &\lesssim 2^{2j\left(d+3-\frac{1}{q}\right)}N^{a_{\kappa,l}+\max\left\{\frac{1}{q},\frac{5-d}{(d-1)q}\right\}}\|G\|_{L_t^{q'}L_x^1}, 
    \end{align*}
    hence
    \begin{align*}
        \|(V_{N,j}^{l,\beta}(\cdot))^\ast G\|^2_{l^2_{n\sim N}L^2_{x'}}&\lesssim 2^{2j\left(d+3-\frac{1}{q}\right)}N^{a_{\kappa,l}+\max\left\{\frac{1}{q},\frac{5-d}{(d-1)q}\right\}}\|G\|^2_{L_t^{q'}L_x^1},
    \end{align*}
    which implies that
     \begin{align*}
        \|(V_{N,j}^{l,\beta}(\cdot))^\ast G\|_{l^2_{n\sim N}L^2_{x'}}&\lesssim 2^{j\left(d+3-\frac{1}{q}\right)}N^{\frac{a_{\kappa,l}}{2}+\max\left\{\frac{1}{2q},\frac{5-d}{2(d-1)q}\right\}}\|G\|_{L_t^{q'}L_x^1}.
    \end{align*}
    This in turn by duality gives
    \begin{align*}
        \|(V_{N,j}^{l,\beta}(\cdot)) F\|_{L_t^qL_x^\infty}&\lesssim 2^{2j\left(\frac{d+3}{2}-\frac{1}{2q}\right)}N^{\frac{a_{\kappa,l}}{2}+\max\left\{\frac{1}{2q},\frac{5-d}{2(d-1)q}\right\}}\|F_n\|_{l_{n\sim N}^2L^2_{x'}}.
    \end{align*}
    This means that
    \begin{align*}
        \|T_{l,\beta}u^{\pm}_{N,j}\|_{L_t^qL_x^\infty}&\lesssim 2^{2j\left(\frac{d+3}{2}-\frac{1}{2q}\right)}N^{\frac{a_{\kappa,l}}{2}+\max\left\{\frac{1}{2q},\frac{5-d}{2(d-1)q}\right\}}\|F^{\pm}_{N,j}\|_{l_{n\sim N}^2L^2_{x'}}.
    \end{align*}

    For every positive real number $\rho$, we note the identities
    \begin{align*}
        \int_0^\infty x_d^{\frac{1}{\kappa}-1}|b_n(\rho x_d)|^2\,dx_d&=\rho^{-\frac{1}{\kappa}}\\
        \kappa\int_0^\infty x_d^{\frac{1}{\kappa}}(|\partial_db_n(\rho x_d)|^2+\rho^2|b_n(\rho x_d)|^2)\,dx_d&=\mu_n\rho^{-\frac{1}{\kappa}}.
    \end{align*}
    Let $\rho=|\xi'|$. It then immediately follows that
    \begin{align*}
        E_{N,j}&=\sum_{\substack{n\sim N}}\int_{|\xi'|\sim 2^{2j}}\int_0^\infty x_d^{\frac{1}{\kappa}-1}|b_n(|\xi'|x_d)|^2|\widehat{g}_{n,j}(\xi')|^2 dx_d\,d\xi'\\
        &+\sum_{\substack{n\sim N}}\kappa\int_{|\xi'|\sim 2^{2j}}\int_0^\infty x_d^{\frac{1}{\kappa}}((|\partial_db_n(|\xi'|x_d)|^2+|\xi'|^2|b_n(|\xi'|x_d)|^2)|\widehat{f}_{n,j}(\xi')|^2 dx_d\,d\xi'\\
        &=\sum_{\substack{n\sim N}}\int_{|\xi'|\sim 2^{2j}}|\xi'|^{-\frac{1}{\kappa}}|\widehat{g}_{n,j}(\xi')|^2+\mu_n|\xi'|^{1-\frac{1}{\kappa}}|\widehat{f}_{n,j}(\xi')|^2\,d\xi'\\
        &=\sum_{\substack{n\sim N}}\int_{|\xi'|\sim 2^{2j}}\mu_n|\xi'|^{1-\frac{1}{\kappa}}|\widehat{F^+_{n,j}}(\xi')-\widehat{F^-_{n,j}}(\xi')|^2+\mu_n|\xi'|^{1-\frac{1}{\kappa}}|\widehat{F^+_{n,j}}(\xi')+\widehat{F^-_{n,j}}(\xi')|^2\,d\xi'\\
        &=2\sum_{\substack{n\sim N}}\int_{|\xi'|\sim 2^{2j}}\mu_n|\xi'|^{1-\frac{1}{\kappa}}(|\widehat{F^+_{n,j}}(\xi')|^2+|\widehat{F^-_{n,j}}(\xi')|^2)\,d\xi'\\
        &\simeq_{\kappa}N2^{2j\left(1-\frac{1}{\kappa}\right)}(\|F^+_{N,j}\|^2_{l_{n\sim N}^2L_{x'}^2}+\|F^-_{N,j}\|^2_{l_{n\sim N}^2L_{x'}^2})
    \end{align*}
    In particular,
    \begin{align*}
      \|F^{\pm}_{N,j}\|_{l_{n\sim N}^2L_{x'}^2}&\lesssim N^{-\frac{1}{2}}2^{2j\left(\frac{1}{2\kappa}-\frac{1}{2}\right)}E_{N,j}^{\frac{1}{2}}.   
    \end{align*}
    This now readily implies the estimate
    \begin{align*}
        \|T_{l,\beta}u_{N,j}\|_{L_t^qL_x^\infty}&\lesssim 2^{2j\left(\frac{d+2}{2}+\frac{1}{2\kappa}-\frac{1}{2q}\right)}N^{\frac{a_{\kappa,l}}{2}+\max\left\{\frac{1}{2q},\frac{5-d}{2(d-1)q}\right\}-\frac{1}{2}}E_{N,j}^{\frac{1}{2}}.
    \end{align*}
    Let $\displaystyle s_c(q):=\frac{d}{2}+\frac{1}{2\kappa}+1-\frac{1}{2q}$.
   It follows that
    \begin{align*}
        \|T_{l,\beta}u_{N,j}\|_{L_t^qL_x^\infty}&\lesssim (2^{2j}N)^{s_c(q)-s}N^{\frac{a_{\kappa,l}}{2}-\frac{1}{2\kappa}+\max\left\{\frac{1}{q},\frac{2}{(d-1)q}\right\}-\frac{d+3}{2}}(2^{2j}N)^{s}E_{N,j}^{\frac{1}{2}}.
    \end{align*}
    We note that $\frac{a_{\kappa,l}}{2}-\frac{1}{2\kappa}+\max\left\{\frac{1}{q},\frac{2}{(d-1)q}\right\}-\frac{d+3}{2}\leq \max\left\{\frac{1}{q},\frac{2}{(d-1)q}\right\}-\frac{d-1}{2}\leq -\delta_{q,d}$, where $\delta_{q,d}=\frac{d-1}{4}$ for $d\geq 3$, and $\delta_{q,d}>0$ is just a fixed positive number when $d=2$, since we must necessarily have $q>4$. Here we have used the admissibility condition $\frac{1}{q}\leq \frac{d-1}{4}$. By the triangle inequality, we now get that
    \begin{align*}
\|T_{l,\beta}u_{N,j}\|_{L_t^qL_x^\infty}&\lesssim\sum_{\substack{N\geq 1,j\geq 0\\N\in 2^{\mathbb{Z}}}}\|T_{l,\beta}u\|_{L_t^qL_x^\infty}\lesssim \sum_{\substack{N\geq 1,j\geq 0\\N\in 2^{\mathbb{Z}}}}(2^{2j}N)^{s_c(q)-s}N^{-\delta_{q,d}}(2^{2j}N)^{s}E_{N,j}^{\frac{1}{2}}\\
      &\lesssim \left(\sum_{\substack{N\geq 1,j\geq 0\\N\in 2^{\mathbb{Z}}}}(2^{2j}N)^{2s_c(q)-2s}N^{-\delta_{q,d}}\right)^{\frac{1}{2}}\left(\sum_{\substack{N\geq 1,j\geq 0\\N\in 2^{\mathbb{Z}}}}(2^{2j}N)^{2s}E_{N,j}\right)^{\frac{1}{2}}\\
      &\lesssim_{s,d,q}\|(u_1,\nabla u_0)\|_{\mathcal{H}^{2s}}
    \end{align*}
    
    Here at the last step we have used the dyadic summability of the terms in the first paranthesis, along with Lemma \ref{Energy dyadic blocks lemma}. This immediately implies that
    \begin{align*}
\|\nabla_x^2u\|_{L_t^qL_x^\infty}&\lesssim_{s,d,q}\|(u_1,\nabla u_0)\|_{\mathcal{H}^{2s}},
    \end{align*}
    and finishes the proof in this case. Finally, in the forbidden case $q=4$ and $d=2$, let $q_1>4$ be sufficiently close to $4$ so that $s>s_c(q_1)$. For every finite time $T_0>0$, we then have
    \begin{align*}
\|\nabla_x^2 u\|_{L_t^4L_x^\infty}&\lesssim T_0^{\frac{1}{4}-\frac{1}{q_1}}\|\nabla_x^2 u\|_{L_t^{q_1}L_x^\infty}\lesssim T_0^{\frac{1}{4}-\frac{1}{q_1}}\|(u_1,\nabla u_0)\|_{\mathcal{H}^{2s}},
    \end{align*}
    as desired.
\end{proof}
\section{Construction of our counterexamples to the Strichartz estimates}\label{s:Construction of our counterexamples to the Strichartz estimates}

In this section, we construct the functions that are going to serve as our counterexamples to the classical Strichartz for the wave equation.

 Let $\displaystyle a\in C_0^\infty(\mathbb{R}^d)$ be an even nonnegative function with 
\begin{align*}\supp\, a \in\left\{\tilde{\xi'}|1-\varepsilon(\kappa)\leq|\tilde{\xi'}|\leq 1+\varepsilon(\kappa)\right\} 
\end{align*}
which is not identically zero.

For $j\geq 0$, we define
\begin{align*}
        U^j(t,x_d,x')&=\frac{1}{(2\pi)^{d-1}}\int e^{ix'\cdot\xi'+it|\xi'|^{\frac{1}{2}}}B\left(0,|\xi'|x_d\right)a(2^{-2j}\xi')\,d\xi' 
    \end{align*}
We note that
\begin{align*}
    U^j(t,x_d,x')&=\frac{1}{(2\pi)^{d-1}}\int e^{ix'\cdot\xi'+it|\xi'|^{\frac{1}{2}}}B\left(0,|\xi'|x_d\right)a(2^{-2j}\xi')\,d\xi'\\
    &=\frac{1}{(2\pi)^{d-1}}\int e^{ix'\cdot\xi'+it|\xi'|^{\frac{1}{2}}}e^{-|\xi'|x_d}a(2^{-2j}\xi')\,d\xi'\\
    &=\frac{2^{2j(d-1)}}{(2\pi)^{d-1}}\int e^{i2^{2j}x'\cdot\eta-2^{2j}|\eta|x_d+i2^jt|\eta|^{\frac{1}{2}}}a(\eta)\,d\eta\\
    &=2^{2j(d-1)}U^0(2^{j}t,2^{2j}x)
\end{align*}
It is clear that $U^j$ is an exact solution to our wave equation
\begin{align*}
    (\partial_t^2-\kappa x_d\Delta_{x'}-\kappa x_d\partial_d^2-\partial_d)U^j&=0
\end{align*}

\begin{proposition}\label{Wave packet norm}
 Let $U^j$ be defined as above. Then, there exists a constant $s_0>0$, uniform in $j$, such that for $0\leq t\leq s_02^{-j}$, we have $\displaystyle \|U^j(t,\cdot)\|_{L_x^r(\Omega)}\gtrsim (2^{2j})^{d-1-\frac{d}{r}}$, uniformly in $t$. Moreover, $\displaystyle \|U^j\|_{\mathcal{H}}\approx 2^{2j\left(\frac{d}{2}-\frac{1}{2\kappa}\right)}$.
\end{proposition}
\begin{proof}
We have
\begin{align*}
    U^0(0,0)&=\frac{1}{(2\pi)^{d-1}}\int a(\eta)\,d\eta>0.
\end{align*}
By continuity, there exist uniform constants $c,C,s_0>0$ such that
\begin{align*}
    |U^0(s,y)|\geq c,
\end{align*}
whenever $t2^j\leq s_0$, and $2^{2j}|x|\leq C$.

This immediately implies that for $t\lesssim 2^{-j}$ and $|x|\lesssim 2^{-2j}$, 
\begin{align*}
    |U^j(t,x)|\gtrsim 2^{2j(d-1)},
\end{align*}
hence
\begin{align*}
 \|U^j(t,x)\|_{L_x^r}\gtrsim \|U^j(t,x)\|_{L_x^r(\{|x|\lesssim 2^{-2j}\})}\gtrsim 2^{2j(d-1)}2^{-2j\frac{d}{r}}.    
\end{align*}
We also estimate the $\mathcal{H}$-norm of the initial data.
\begin{align*}
    U^j(0)&=\frac{1}{(2\pi)^{d-1}}\int e^{ix'\cdot\xi'}e^{-|\xi'|x_d}a(2^{-2j}\xi')\,d\xi' \\
   \partial_tU^j(0)&=\frac{i}{(2\pi)^{d-1}}\int e^{ix'\cdot\xi'}e^{-|\xi'|x_d}|\xi'|^{\frac{1}{2}}a(2^{-2j}\xi')\,d\xi'
\end{align*}
 We have
\begin{align*}
\|x_d^{\frac{1}{2\kappa}-\frac{1}{2}}\partial_tU^j(0,x_d,x')\|_{L_x^2(\Omega)}&=2^{2j(d-1)}2^{j}\|x_d^{\frac{1}{2\kappa}-\frac{1}{2}}\partial_tU^0(0,2^{2j}x_d,2^{2j}x')\|_{L_x^2(\Omega)}\\
&=2^{j}2^{2j(d-1)}2^{2j\left(\frac{1}{2}-\frac{1}{2\kappa}\right)}\|(2^{2j}x_d)^{\frac{1}{2\kappa}-\frac{1}{2}}\partial_tU^0(0,2^{2j}x_d,2^{2j}x')\|_{L_x^2(\Omega)}\\
&=2^{j}2^{2j(d-1)}2^{2j\left(\frac{1}{2}-\frac{1}{2\kappa}\right)}2^{-jd}\|(y_d)^{\frac{1}{\kappa}-\frac{1}{2}}\partial_tU^0(0,y_d,y')\|_{L_y^2(\Omega)}\\
&\approx 2^{2j\left(\frac{d}{2}-\frac{1}{2\kappa}\right)}
\end{align*}
Similarly,
\begin{align*}
    \|x_d^{\frac{1}{2\kappa}}\nabla_{x}U^j(0,x_d,x')\|_{L_x^2(\Omega)}&\approx 2^{2j\left(\frac{d}{2}-\frac{1}{2\kappa}\right)},
\end{align*}
hence
\begin{align*}
    \|U^j(0)\|_{\mathcal{H}}&\approx 2^{2j\left(\frac{d}{2}-\frac{1}{2\kappa}\right)}.
\end{align*}
This finishes the proof.

\end{proof}

\section{Proof of Theorems \ref{Counterexample Strichartz} and \ref{Counterexample Strichartz 2}}\label{s:Proof of Theorems 1 and 2}
We define
\begin{align*}
    \psi^j=\frac{U^j}{2^{2j\left(s+\frac{d}{2}-\frac{1}{2\kappa}\right)}};
\end{align*}
For a wave-admissible pair $\displaystyle(q,r)$, we have
\begin{align*}
     \|\psi^j(t,\cdot)\|_{L_t^qL_x^r([0,1]\times\Omega)}\gtrsim  2^{2j\left(d-1-\frac{d}{r}+\frac{1}{2\kappa}-\frac{d}{2}-\frac{1}{2q}-s\right)}\approx 2^{2j\left(\frac{1}{2q}+\gamma+\frac{1}{2\kappa}-1-s\right)},
 \end{align*}
 hence
 \begin{align*}
    \|\nabla_x^2\psi^j(t,\cdot)\|_{L_t^qL_x^r([0,1]\times\Omega)}\gtrsim  2^{2j\left(d+1-\frac{d}{r}+\frac{1}{2\kappa}-\frac{d}{2}-\frac{1}{2q}-s\right)}\approx 2^{2j\left(\frac{1}{2q}+\gamma+\frac{1}{2\kappa}+1-s\right)}. 
 \end{align*}
It is also clear that
\begin{align*}
    \|(\partial_t\psi^j_0,\nabla_{x}\psi^j_0)\|_{\mathcal{H}}&\lesssim 2^{-2js}. 
\end{align*}
The definition of the spaces $\mathcal{H}^{2s}$, along with interpolation and localizing at the tangential frequency $|\xi'|\simeq 2^{2j}$ also implies that
\begin{align*}
    \|(\partial_t\psi^j_0,\nabla_{x}\psi^j_0)\|_{\mathcal{H}^{2s}}&\lesssim  1. 
\end{align*}

Therefore, whenever $\displaystyle \alpha\in \left(0,\frac{1}{2q}+\gamma+\frac{1}{2\kappa}+1-s\right)$,
we have
\begin{align*}
    2^{-2j\alpha}\|\nabla^2\psi^j(t,\cdot)\|_{L_t^qL_x^r([0,1]\times\Omega)}&\gg\|(\partial_t\psi^j_0,\nabla_{x}\psi^j_0)\|_{\mathcal{H}^{2s}},
\end{align*}
which finishes the proof of Theorem \ref{Counterexample Strichartz}.

When $(q,r)$ is instead Euler-admissible, we have
\begin{align*}
     \|\psi^j(t,\cdot)\|_{L_t^qL_x^r([0,1]\times\Omega)}&\approx  2^{2j\left(d-1-\frac{d}{r}+\frac{1}{2\kappa}-\frac{d}{2}-\frac{1}{2q}-s\right)}\approx 2^{2j\left(-\gamma-s\right)}\\
     \|\nabla^2\psi^j(t,\cdot)\|_{L_t^qL_x^r([0,1]\times\Omega)}&\gtrsim \|\nabla_{x'}^2\psi^j(t,\cdot)\|_{L_t^qL_x^r([0,1]\times\Omega)}\gtrsim 2^{4j}\|\psi^j(t,\cdot)\|_{L_t^qL_x^r([0,1]\times\Omega)}\simeq 2^{2j\left(2-\gamma-s\right)},
 \end{align*}
 where for the last part we have used Bernstein's inequality.
 
 Thus, whenever $\displaystyle \alpha\in \left(0,2-\gamma-s\right)$,
we have
\begin{align*}
    2^{-2j\alpha}\|\nabla^2\psi^j(t,\cdot)\|_{L_t^qL_x^r([0,1]\times\Omega)}&\gg\|(\partial_t\psi^j_0,\nabla_{x}\psi^j_0)\|_{\mathcal{H}^{2s}},
\end{align*}
We also note that, when $\displaystyle d\geq 3$, and $\displaystyle(q,r)=(2,\infty)$, we need at least $\displaystyle s\geq k_0+\frac{1}{4}$ (hence $\displaystyle 2s>2k_0+1$) and this suggests that Strichartz estimates can only potentially be used in order to control $\displaystyle\|\nabla_x^2\psi^j\|_{L_t^2L_x^\infty([0,T]\times\Omega)}$ and improve Ifrim and Tataru's local well-posedness result from \cite{Ifrim-Tataru} only up to $2s\geq 2k_0+\frac{1}{2}$.

\bibliography{compressible_euler_strichartz.bib}
\bibliographystyle{plain}
\end{document}